\documentclass[a4paper,11pt]{article}
\usepackage{makeidx}
\usepackage[mathscr]{eucal}
\usepackage{color}
\usepackage{amssymb, amsfonts, amsmath, amsthm, graphics}

\newtheorem{proposition}{Proposition}[section]
\newtheorem{theorem}{Theorem}[section]
\newtheorem{lemma}[proposition]{Lemma}
\newtheorem{remark}{Remark}[section]
\newtheorem{corollary}[theorem]{Corollary}

\numberwithin{equation}{section}

\title{
Uniqueness and nondegeneracy of ground states to nonlinear scalar field equations 
involving the Sobolev critical exponent in their nonlinearities for high frequencies
}
\author{Takafumi Akahori, Slim Ibrahim, Norihisa Ikoma, Hiroaki Kikuchi, Hayato Nawa}
\date{\today}
\newcommand{\Go}{\mathcal{G}_\omega}
\newcommand{\Gon}{\mathcal{G}_{\omega_n}}
\newcommand{\TGo}{\widetilde{\mathcal{G}}_\omega}
\newcommand{\TGon}{\widetilde{\mathcal{G}}_{\omega_n}}
\newcommand{\R}{\mathbb{R}}
\newcommand{\Rd}{\mathbb{R}^d}
\newcommand{\Pho}{\Phi_\omega}
\newcommand{\Phon}{\Phi_{\omega_n}}
\newcommand{\TPho}{\widetilde{\Phi}_\omega}
\newcommand{\TPhon}{\widetilde{\Phi}_{\omega_n}}

\newcommand{\TPhnone}{\widetilde{\Phi}_{n,1}}
\newcommand{\TPhntwo}{\widetilde{\Phi}_{n,2}}
\newcommand{\dH}{\dot{H}^1}
\newcommand{\dHRd}{\dot{H}^1(\Rd)}
\newcommand{\e}{\varepsilon}

\def\wt#1{\widetilde{#1}}

\newcommand{\Hrad}{H^1_{\rm rad}}
\begin{document}
\maketitle

\begin{abstract}
The study of the uniqueness and nondegeneracy of ground state solutions to 
semilinear elliptic equations is of great importance because of the resulting energy landscape and its implications 
for the various dynamics. In \cite{AIKN3}, semilinear elliptic equations with combined power-type nonlinearities involving the Sobolev critical exponent are studied. There, it is shown that if the dimension is four or higher, 
and the frequency is sufficiently small, then the positive radial ground state is unique and nondegenerate. 
In this paper, we extend these results to the case of high frequencies when the dimension is five and higher.
After suitably rescaling the equation, we demonstrate that the main behavior of the solutions is given by the Sobolev critical part for which the ground states are explicit, and their degeneracy is well characterized. Our result is a key step towards the study of the different dynamics of solutions of the corresponding nonlinear Schr\"odinger and Klein-Gordon equations with energies above the energy of the ground state. Our restriction on the dimension is mainly due to the existence of resonances in dimension three and four. 
\end{abstract}

\section{Introduction}\label{section:1}

In this paper, we investigate the uniqueness and nondegeneracy of ground state 
to the semilinear elliptic equation of the form 
\begin{equation}\label{eq:1.1} 
  - \Delta u + \omega u
=
|u|^{p-1}u+|u|^{\frac{4}{d-2}}u 
\quad 
\mbox{in $\mathbb{R}^{d}$}, \quad u \in H^1(\Rd) :=H^1(\Rd, \mathbb{C})
\end{equation}
where $d\ge 3$, $\omega >0$ and $1< p < \frac{d+2}{d-2} = : 2^\ast - 1$. 
Here we call $u$ a \emph{ground state} to \eqref{eq:1.1} 
provided $u$ is a nontrivial solution to \eqref{eq:1.1} satisfying 
	\[
		\mathcal{S}_\omega (u) 
		= \inf \left\{ \mathcal{S}_\omega(v) : 
		\text{$v$ is a nontrivial solution to \eqref{eq:1.1}} \right\}
	\]
where the action $S_\omega$ is defined by
\begin{equation}\label{eq:1.2}
\mathcal{S}_{\omega}(u)
:=\frac{1}{2}\|\nabla u \|_{L^{2}}^{2}
+
\frac{\omega}{2}\|u\|_{L^{2}}^{2}
-
\frac{1}{p+1}\|u\|_{L^{p+1}}^{p+1}
-
\frac{1}{2^{*}}\|u\|_{L^{2^{*}}}^{2^{*}}.
\end{equation}
Observe that critical points of $\mathcal{S}_\omega$ solve \eqref{eq:1.1}. 
In addition, a solution $u$ to \eqref{eq:1.1} is said to be \emph{nondegenerate in $\Hrad(\Rd)$} 
when the linearized equation of \eqref{eq:1.1} at $u$ 
	\[
		-\Delta \varphi + \omega \varphi = 
		p |u|^{p-1} \varphi + \frac{d+2}{d-2} |u|^{\frac{4}{d-2}} \varphi 
		\quad {\rm in}  \ \Rd, \quad \varphi \in 
		\Hrad (\Rd) := \{ u \in H^1(\Rd) : \text{$u$ is radial} \} 
	\]
has the trivial function $\varphi\equiv0$ as its unique solution. That is
	\[
		{\rm Ker}\, L_u|_{\Hrad(\Rd)} = \left\{ 0 \right\}
	\]
where
$L_u : H^1(\Rd) \to (H^1(\Rd))^\ast $ is defined by 
	\begin{equation}\label{eq:1.3}
				L_u := -\Delta + \omega 
				- p |u|^{p-1} - \frac{d+2}{d-2} |u|^{\frac{4}{d-2}}. 
	\end{equation}
Equation \eqref{eq:1.1} appears in studying standing wave solutions 
for the following nonlinear Schr\"odinger equation and Klein-Gordon equation:
\begin{align}
\label{eq:1.4} 
i \frac{\partial \psi}{\partial t}  - \Delta \psi 
=
|\psi|^{p-1}\psi+|\psi|^{\frac{4}{d-2}}\psi ,
\\[6pt]
\label{eq:1.5} 
\frac{\partial^{2} \psi}{\partial t^{2}}  - \Delta \psi + m \psi
=
|\psi|^{p-1}\psi+|\psi|^{\frac{4}{d-2}}\psi
. 
\end{align}
More precisely, when we look for solutions of the form $\psi(t,x) = e^{-i \lambda t} u(x)$ 
($\lambda \in \R$), we observe that \eqref{eq:1.4} or \eqref{eq:1.5} is equivalent 
to solving \eqref{eq:1.1} with the choices $\lambda=\omega$ and $\lambda = \pm \sqrt{m-\omega}$, $\omega<m$. 

Due to the presence of multiple powers in \eqref{eq:1.1},  \eqref{eq:1.4} or \eqref{eq:1.5}, these equations loose their scaling invariances and thus a main interest in studying them is to explore the implications of such perturbations, in particular the emergence of ground state solitary waves, the resulting energy landscape, 
and its role for the global dynamics.

Recently, the dynamics for \eqref{eq:1.4} and \eqref{eq:1.5} were intensively studied. When the energy of initial data is less than the ground state energy, only two scenarios can happen: finite time blow-up or scattering. For example, we refer to  \cite{Akahori-Ibrahim-Kikuchi-Nawa1, AIKN3, Kenig-Merle, Killip-Oh-Poco-Visan, MXZ}. However, when the energy of initial data is slightly greater than the ground state energy, the dynamic is much more complicated, and the combination of finite time blow-up, scattering and non-dispersion behaviors are shown in forward or backward in time. We refer to \cite{AIKN3,Nakanishi-Schlag1, Nakanishi-Schlag2} for more details.
In studying the dynamics around the ground state, 
basic properties of ground state such as the uniqueness and nondegeneracy 
play a crucial role. 
This is our main motivation for the present paper.

	On the other hand, the uniqueness and nondegeneracy of positive solutions have 
applications to the study of nonlinear elliptic equations. 
For instance, let us consider the following singular perturbation problem
	\begin{equation}\label{eq:1.6}
		-\e^2 \Delta v + V(x) v = f(v) \quad {\rm in} \ \Rd
	\end{equation}
where $V(x): \Rd \to \R$ and $f: \R \to \R$ are given functions and 
$0 < \e \ll 1$ a parameter. When we try to find spike solutions to \eqref{eq:1.6} 
concentrating at some point $x_0 \in \Rd$, the uniqueness and nondegeneracy of positive solutions to 
	\[
		-\Delta u + V(x_0) u = f(u) \quad {\rm in} \ \Rd
	\]
are keys in order to apply the Lyapunov--Schmidt reduction method. 
For instance, see \cite{FW-86,O-88,AM}. 
Since it suffices to consider positive ground states due to Proposition \ref{proposition:1.2} for \eqref{eq:1.1}, 
the uniqueness and nondegeneracy of ground states to \eqref{eq:1.1} 
is regarded as a step toward those of positive solutions to \eqref{eq:1.1}. 
Therefore, to study those properties is fundamental and important.

In the case of a single power nonlinearity, and in his celebrated paper \cite{Kwong}, Kwong proved 
the uniqueness and nondegeneracy of positive radial solutions to 
	\begin{equation}\label{eq:1.7}
		-\Delta u + u = u^p \quad {\rm in} \ \Rd, \quad u(x) \to 0 \quad {\rm as} \ |x| \to \infty
	\end{equation}
where $d \geq 1$ and $ 1 < p < 2^\ast - 1$. 
For generalizations and related results, 
we refer to \cite{CL-91,C,KZ91,MS,PS, ST} and references therein. 
Here it is important to mention that it is not simple to apply those results for \eqref{eq:1.1} except for 
some particular cases. 
In fact, to the best of our knowledge, Pucci and Serrin \cite{PS} studied 
the uniqueness of radial positive solutions to 
	\[
		\Delta u + f(u) = 0 \quad \mbox{in $\mathbb{R}^{d}$}, 
		\quad u(x) \to 0 \quad {\rm as} \ |x| \to \infty
	\]
and treated a general nonlinearity $f(u)$. 
We will see in Appendix \ref{section:C} that 
the result of \cite{PS} can be applied to \eqref{eq:1.1} 
when $3 \leq d \leq 6$ and $2^\ast - 2 \leq p < 2^\ast - 1$ with $1<p$, and 
unfortunately, not in the case $d \geq 7$ and $\omega \gg 1$. 
See Proposition \ref{proposition:C.1} and Remark \ref{remark:C.1}, 
and for other cases, we do not know whether or not 
the result of \cite{PS} can be applied. 
Furthermore, the nondegeneracy of radial solutions is not treated in \cite{PS}.

	In addition, the uniqueness of radial positive solutions to \eqref{eq:1.1} is delicate according to \cite{DdPG}. 
In \cite{DdPG}, D\'avila, del Pino and Guerra gave a numerical result 
which shows that the uniqueness of positive solutions to \eqref{eq:1.1} fails 
for $d=3$, $1<p<2^\ast - 2$ and $\omega \ll 1$.

	About the uniqueness and nondegeneracy of ground states to \eqref{eq:1.1}, 
these properties were proved in \cite{AIKN3} 
under the assumptions $d\ge 4$, $1+\frac{4}{d}<p<2^\ast - 1$ and 
$\omega \ll 1$. 
We also mention that the papers \cite{Grossi,GLP} studied  
the uniqueness and nondegeneracy of ground states 
to equations in bounded domains with single power type nonlinearity 
whose exponent is the critical one or close to it.

	Recently, Coles and Gustafson \cite{ColesGustafson} showed the uniqueness of 
the ground state to \eqref{eq:1.1} when $d=3$, $3<p<5$ and $\omega \gg 1$. 
For more precise statement, see Remark \ref{remark:1.2}. 
Here we also note that they also study the dynamics 
of the perturbed critical nonlinear Schr\"odinger equation.

From the above observations, our aim in this paper is to address the uniqueness and nondegeneracy of ground state 
to \eqref{eq:1.1} for $\omega \gg 1$ in the higher dimensional case. To state our result more precisely, 
we first recall the existence of ground state to \eqref{eq:1.1}:

\begin{proposition}[\cite{ASM, Zhang-Zou} (cf. \cite{Brezis-Nirenberg})]\label{proposition:1.1}
Assume either $d = 3$ and $3<p<5$ or else $d\ge4$ and $1<p<\frac{d+2}{d-2}$. 
Then, for any $\omega >0$ there exists a ground state to \eqref{eq:1.1}. 
\end{proposition}

For the sake of clarity and self-content,  a sketch of the proof of Proposition \ref{proposition:1.1} will be given 
	in Appendix \ref{section:A}, using simpler arguments than those in \cite{ASM, Zhang-Zou}. 

\begin{remark}\label{remark:1.1}
In \cite[Theorem 1.2]{Akahori-Ibrahim-Kikuchi-Nawa1}, 
 the nonexistence of ground state to \eqref{eq:1.1} was mistakenly claimed for $d=3$, 
$1+\frac{4}{d}< p < 2^\ast - 1$ and sufficiently large $\omega$.  Indeed, in the proof of that Theorem, \emph{(3.18)} was overlooked and now the above Proposition \ref{eq:1.1} fixes that. 
\end{remark}

Using a standard argument for semilinear elliptic equation (see \cite{AIKN3, GNN, Lieb-Loss}), 
we can derive the following properties of the ground states to \eqref{eq:1.1}:

\begin{proposition}\label{proposition:1.2}
Assume either $d = 3$ and $3<p<5$ or else $d\ge4$ and $1<p<\frac{d+2}{d-2}$. 
Then, for any $\omega >0$ and any ground state $Q_{\omega}$ to \eqref{eq:1.1}, the following properties hold:
\begin{enumerate}
\item[\emph{(i)}] $Q_\omega \in C^2(\Rd,\mathbb{C})$. 
\item[\emph{(ii)}] There exist $y \in \Rd$, $\theta \in \R$ and a radial positive function 
$\Phi_\omega$ such that $Q_\omega (x) = e^{i \theta} \Phi_\omega (x-y)$. 
In particular, $\Phi_\omega$ is a radial positive ground state to \eqref{eq:1.1}. 
\item[\emph{(iii)}] Each radial positive ground state to \eqref{eq:1.1} is 
strictly decreasing in the radial direction.
\end{enumerate}
\end{proposition}

From Proposition \ref{proposition:1.2}, it suffices to study 
radial positive ground states to \eqref{eq:1.1}. 
Now, we state our main result: 

\begin{theorem}\label{theorem:1.1}
Assume $d\ge 5$ and $1 <p < \frac{d+2}{d-2}$. 
Then, there exists an $\omega_{*}>0$ such that for any $\omega>\omega_{*}$, 
 the (radial) positive ground state to \eqref{eq:1.1} is unique and 
nondegenerate in $\Hrad(\Rd)$. 
\end{theorem}

\begin{remark}\label{remark:1.2}
{\rm (i)} When $3 \leq d \leq 6$ and $\frac{4}{d-2} \leq p < \frac{d+2}{d-2}$ with $1<p$, 
by Propositions \ref{proposition:1.1}, \ref{proposition:1.2} and \ref{proposition:C.1}, equation
\eqref{eq:1.1} admits a unique radial positive solution for any $\omega > 0$. 
Furthermore, combining this fact with Theorem \ref{theorem:1.1}, 
when $d=5,6$, $\frac{4}{d-2} \leq p < \frac{d+2}{d-2}$ with $1<p$ and $\omega > \omega_\ast$, 
we find that 
the unique positive solution to \eqref{eq:1.1} is nondegenerate in $\Hrad(\Rd)$. 

\noindent
{\rm (ii)}
It follows from \cite[Section 5]{Ni-Takagi} and Theorem \ref{theorem:1.1} that for any $d\ge 5$, any $\omega > \omega_\ast$ and the radial positive ground state $\Phi_\omega$ to \eqref{eq:1.1}, 
we can prove
	\[
		{\rm Ker}\, L_{\Pho} 
		= \left\{ \sum_{j=1}^d a_j \frac{\partial \Phi_\omega}{\partial x_j} : 
		a_j \in \mathbb{C} 
		 \right\}. 
	\]
See \eqref{eq:1.3} for the definition of $L_{\Pho}$. 

\noindent
{\rm (iii)} In \cite{ColesGustafson}, Coles and Gustafson established the uniqueness of the ground state to 
	$$
		 -\Delta u + \hat{\omega}(\e) u =\varepsilon |u|^{p-1}u+|u|^{\frac{4}{d-2}}u \quad {\rm in} \ \R^3
	$$
where $3<p<5$, $0<\e \ll 1$, $\hat{\omega}(\e) = \omega_1 \e^2 + o(\e^2)$ and $\omega_1>0$. 
By scaling, the equation can be rewritten as  
	\[
		-\Delta u + \omega_{\e} u  = |u|^{p-1} u + |u|^{\frac{4}{d-2}} u \quad {\rm in} \ \R^3, \quad 
		\omega_{\e} \to \infty \quad {\rm as} \ \e \to 0.
	\]
Therefore, Theorem \ref{theorem:1.1} is a higher-dimensional counterpart of the results 
in \cite{ColesGustafson}, however, 
our argument is different from the one in \cite{ColesGustafson}. 
\end{remark}

Next, we describe the difficulties and ideas of the proof of Theorem \ref{theorem:1.1} 
	as well as its comparison with the case $\omega \ll 1$ and the result of \cite{Grossi,GLP}.  
Set 
	\[
		\Go := \left\{ \Pho  : 
		\text{$\Pho$ is a radial positive ground state to \eqref{eq:1.1}}. \right\}
	\]
	
Our aim is to show that $\Go$ is a singleton and 
${\rm Ker}\, L_{\Pho}|_{\Hrad (\Rd)}  = \{0\}$ for $\Pho \in \Go$.

In \cite{AIKN3}, for $\omega \ll 1$ and  $\Pho \in \Go$, 
we use the following rescaling 
corresponding to the subcritical power $p$: 
\[
T_{\omega}[\Phi_{\omega}] 
:=\omega^{-\frac{1}{p-1}}\Phi_{\omega} \left( \omega^{-\frac{1}{2}} x  \right),
\]
so that $T_{\omega} [ \Pho ]$ solves 
\[
-\Delta u + u  
-
|u|^{p-1}u
- 
\omega^{\frac{2^{*}-(p+1)}{p-1}}|u|^{\frac{4}{d-2}}u 
=0.
\]
Next, we showed that $T_\omega [\Pho] \to U$ 
strongly in $H^1(\Rd)$ as $\omega \to 0$ where 
$U$ is a unique radial positive solution to \eqref{eq:1.7}. 
By \cite{Kwong}, we know that $U$ is nondegenerate in $\Hrad(\Rd)  $, that is, 
\[
{\rm Ker}\, L^{\dagger}_{U} |_{\Hrad (\Rd)} = \left\{ 0 \right\}, \quad 
L^{\dagger}_{U} := -\Delta + 1 - p U^{p-1}.
\]
Hence, from the implicit function theorem, we observe that 
for $\omega \ll 1$, 
\eqref{eq:1.1} admits a unique radial positive ground state to \eqref{eq:1.1} 
which is nondegenerate in $\Hrad (\Rd)$.

On the other hand, when $\omega \gg 1$, 
the Sobolev critical term becomes dominant and 
we use the following rescaling (cf. \cite{Han, Moroz-Muratov, Pan-Wang, PS, ST}): 
\begin{equation}\label{eq:1.8}
\widetilde{\Phi}_{\omega}(x)=M_{\omega}^{-1}\Phi_{\omega} \left(M_{\omega}^{-\frac{2}{d-2}}x\right), 
\quad M_\omega := \Pho (0) = \| \Pho \|_{L^\infty(\Rd)}. 
\end{equation}
Then, we can  verify that
\begin{equation}\label{eq:1.9}
-
\Delta \widetilde{\Phi}_{\omega}
+
\alpha_{\omega}  \widetilde{\Phi}_{\omega}
-
\beta_{\omega} \widetilde{\Phi}_{\omega}^{p}
-
\widetilde{\Phi}_{\omega}^{\frac{d+2}{d-2}}
=
0, \quad 
\alpha_{\omega}:=\omega M_{\omega}^{-\frac{4}{d-2}}, \quad 
\beta_{\omega}:= M_{\omega}^{p-1-\frac{4}{d-2}}. 
\end{equation}
Next, we prove $\alpha_\omega, \beta_\omega \to 0$ and 
$\TPho \to W$ strongly in $\dot{H}^1(\Rd)$ as $\omega \to \infty$ 
where 
	\[
		\dot{H}^1(\Rd) := \left\{ u \in L^{2^\ast} (\Rd) : 
		\| u \|_{\dot{H}^1} :=   \| \nabla u \|_{L^2(\Rd)} < \infty \right\}
	\]
and $W$ is \emph{the Talenti function} with $W(0)=1$, that is,  
\begin{equation}\label{eq:1.10}
W(x)
:=\left(1+\frac{|x|^{2}}{d(d-2)} \right)^{-\frac{d-2}{2}}, \quad 
-\Delta W = W^{\frac{d+2}{d-2}} \quad {\rm in} \ \Rd. 
\end{equation}
Remark that the convergence is proved in $\dH(\Rd)$, 
which is different from $H^1(\Rd)$. 
Moreover, in contrast to $U$,  $W$ is degenerate in $\dH_{\rm rad}(\Rd)$, and 
\begin{equation}\label{eq:1.11}
{\rm Ker}\, L^{\ddagger}_W|_{\dH_{\rm rad}(\Rd)} = \left\{ \Lambda W \right\}  
\end{equation}
where 
\begin{equation}\label{11111}
L^{\ddagger}_W := -\Delta + \frac{d+2}{d-2} W^{\frac{4}{d-2}}, \quad 
\Lambda W := \frac{d-2}{2}W+x\cdot \nabla W 
= \frac{d}{d \lambda} \Big|_{\lambda =1}
\lambda^{\frac{d-2}{2}} W( \lambda x ).
\end{equation}
From these facts, we need more delicate analysis to show the uniqueness 
and nondegeneracy for $\omega \gg 1$.

To overcome the above difficulties, we use arguments inspired 
by \cite{Grossi,GLP}. 
We argue indirectly and suppose that Theorem \ref{theorem:1.1} fails to hold. 
To derive a contradiction, 
key ingredients consist of a uniform decay estimate of elements of $\TGo$ and 
Pohozaev's identity corresponding to \eqref{eq:1.9} where 
		\[
		\TGo := 
		\{ \TPho \colon
		\Phi_\omega \in \Go
		\}.
	\]

For the uniform spatial decay, we use the Kelvin transform to reduce the problem to a  ball and apply Moser's iteration scheme. One of differences from \cite{Grossi,GLP} 
is the presence of the subcritical term $\beta_\omega \TPho^p$ and 
we have to show that this term is harmless to employ the Moser iteration.

After showing the uniform decay, we derive a contradiction. 
In \cite{Grossi,GLP}, contradictions are derived through 
the information on boundary data. 
In our case, we investigate the behaviors of 
$\alpha_\omega \TPho$ and $\beta_\omega \TPho^p$ in \eqref{eq:1.9} 
with Pohozaev's identity. 
To this end, we need not only 
the convergence of $\widetilde{\Phi}_{\omega}$ in $\dot{H}^{1}(\mathbb{R}^{d})$ 
but also in $L^{2}(\mathbb{R}^{d})$ which can be obtained from the uniform decay. 
This requires us to assume $d \geq 5$.

Now, we introduce the notation used in this paper. 
By $B_R$ we denote the open ball in $\mathbb{R}^{d}$ of center $0$ and radius $R$, namely, $B_{R}:=\{x\in \mathbb{R}^{d} : |x| < R \}$. 
For given positive quantities $a$ and $b$, the notation $a \lesssim b$ means 
the inequality $a\le C b$ for some positive constant $C$ depending only on $d$ and $p$.

This paper is organized as follows. 
In Section \ref{section:2}, we prove the convergence results of elements of $\TGo$ as $\omega \to \infty$. 
Section \ref{section:3} is devoted to deriving the uniform decay estimate for elements of $\TGo$. 
Finally, in Section \ref{section:4}, we give a proof of Theorem \ref{theorem:1.1}. 
For readers' convenience, we include an Appendix where 
we prove Proposition \ref{proposition:1.1} in Section \ref{section:A}, 
state a result of the Moser iteration technique in Section \ref{section:B}, 
and discuss when the result of \cite{PS} is applicable to \eqref{eq:1.1} in Section \ref{section:C}.

%%%%%%%%%%%%%%%%%%%%%%%%%%%%%%%%%%%%%%%%%%%%%%%%%%%%%%%%%%%%%%%%%%%%%%%%%%%%%%%%
\section{Convergence as $\boldsymbol{\omega\to \infty}$}\label{section:2}

Our aim in this section is to prove the following convergence result of rescaled ground states:
\begin{proposition}\label{proposition:2.1}
	Assume $d\ge 3$, $1<p<\frac{d+2}{d-2}$ and $\Go \neq \emptyset$. Then, we have
	\[
	\lim_{\omega \to \infty} 
	\sup_{\TPho \in \TGo} 
	\big\| \TPho - W  \big\|_{\dH} = 0
	\]
	where $\TPho$ and $W$ are defined in \eqref{eq:1.8} and \eqref{eq:1.10}.
\end{proposition}

In order to prove Proposition \ref{proposition:2.1}, we  introduce Nehari's and Pohozaev's functionals %one $\mathcal{P}_{\omega}$ 
(associated to equation \eqref{eq:1.1}) defined by:  
\begin{align}
\mathcal{N}_{\omega}(u)
&:=
\|\nabla u \|_{L^{2}}^{2}
+
\omega 
\|u \|_{L^{2}}^{2}
-
\|u\|_{L^{p+1}}^{p+1}
-
\|u\|_{L^{2^{*}}}^{2^{*}},
\label{eq:2.1}
\end{align}
and
%\\[6pt]

\begin{align}
\mathcal{P}_{\omega}(u)
&:=
\frac{1}{2^{*}} \|\nabla u \|_{L^{2}}^{2}
+
\frac{\omega}{2} 
\|u \|_{L^{2}}^{2}
-
\frac{1}{p+1}\|u\|_{L^{p+1}}^{p+1}
-
\frac{1}{2^{*}}\|u\|_{L^{2^{*}}}^{2^{*}}
\label{eq:2.2}
\end{align}
respectively.
Recalling \eqref{eq:1.2}, 
the following linear combinations are useful in the study of ground states to \eqref{eq:1.1}: 
\begin{align}
\mathcal{S}_{\omega}(u)- \mathcal{P}_{\omega}(u)
&=
\frac{1}{d}\|\nabla u \|_{L^{2}}^{2},
\label{eq:2.3}
\\[6pt]
\label{eq:2.4}
\mathcal{P}_{\omega}(u)-\frac{1}{2^{*}}\mathcal{N}(u)
&=
\frac{\omega}{d} \|u \|_{L^{2}}^{2}
-
\frac{2^{*}-(p+1)}{2^{*}(p+1)}\|u\|_{L^{p+1}}^{p+1},
\\[6pt]
\label{eq:2.5}
\mathcal{P}_{\omega}(u)-\frac{1}{p+1}\mathcal{N}(u)
&=
\frac{p-1}{2(p+1)} \omega \|u \|_{L^{2}}^{2}
-
\frac{2^{*}-(p+1)}{2^{*}(p+1)}
\Big\{
\|\nabla u\|_{L^{2}}^{2}
-
\|u\|_{L^{2^{*}}}^{2^{*}}
\Big\}
.
\end{align}

We record the following basic properties of solutions to \eqref{eq:1.1}: 
\begin{lemma}\label{proposition:2.2}
Assume $d\ge 3$, $1<p<\frac{d+2}{d-2}$ and $\omega>0$. Then, the following hold:
\begin{enumerate}
\item[{\rm (i)}] If $u$ is an $H^{1}$-solution to \eqref{eq:1.1}, then 
\begin{equation}\label{eq:2.6}
\mathcal{N}_{\omega}(u)=\mathcal{P}_{\omega}(u)
=0, 
\qquad  S_\omega (u) = \frac{1}{d} \| \nabla u \|_{L^{2}}^{2}. 
\end{equation}
\item[{\rm (ii)}] If $\Pho \in \Go$, then 
\begin{align}
\label{eq:2.7}
\|\nabla \Pho \|_{L^{2}}
&\le 
\|\nabla W \|_{L^{2}},
\\[6pt]
\label{eq:2.8}
\|\Pho \|_{L^{2}}
&\lesssim 
\omega^{-\frac{1}{2}},
\\[6pt]
\label{eq:2.9}
|\Pho(x)|
&\lesssim  
\omega^{-\frac{1}{4}} |x|^{-\frac{d-1}{2}} \quad 
{\rm for\ all} \ x \in \Rd \setminus \{0\}
.
\end{align}
\end{enumerate}
\end{lemma}
\begin{proof}
See \cite{Berestycki-Lions} for the proof of the identities in \eqref{eq:2.6}. 
The inequality \eqref{eq:2.7} follows from \eqref{eq:2.6} and \eqref{eq:A.5}. 
The inequality \eqref{eq:2.8} follows from \eqref{eq:2.5}, \eqref{eq:2.6} and \eqref{eq:2.7}.  Finally, we prove \eqref{eq:2.9}. Using the fundamental theorem of calculus, H\"older's inequality and Hardy's inequality, we see that 
\begin{equation*}
\begin{aligned}
|x|^{d-1} |\Phi_{\omega}(x)|^{2}
&\le 
\int_{0}^{|x|} 
\left\{ 
(d-1)r^{d-1} \frac{|\Phi_{\omega}(r)|}{r} |\Phi_{\omega}(r)|
+
2r^{d-1} |\Phi_{\omega}'(r)| |\Phi_{\omega}(r)|
\right\}
\,dr 
\\[6pt]
&\lesssim 
\left(
\int_{0}^{|x|} 
\left\{
\frac{|\Phi_{\omega}(r)|^{2}}{r^{2}} 
+
|\Phi_{\omega}'(r)|^{2} \right\} r^{d-1} \,dr
\right)^{\frac{1}{2}}
\left(
\int_{0}^{|x|} 
|\Phi_{\omega}(r)|^{2} r^{d-1} \,dr
\right)^{\frac{1}{2}}
\\[6pt]
&\lesssim  
\| \nabla \Phi_{\omega}\|_{L^{2}} \|\Phi_{\omega}\|_{L^{2}}
.
\end{aligned}
\end{equation*}
Furthermore, applying the inequalities \eqref{eq:2.7} and \eqref{eq:2.8} to the right-hand side above, we obtain the desired result \eqref{eq:2.9}.  
\end{proof}

Next, we consider the rescaled ground states. Let $\widetilde{\Phi}_{\omega} \in \TGo$. Then, $\widetilde{\Phi}_{\omega}$ satisfies equation \eqref{eq:1.9} and 
\begin{equation}\label{eq:2.10}
\|\widetilde{\Phi}_{\omega}\|_{L^{\infty}}=\widetilde{\Phi}_{\omega}(0)=1=\|W\|_{L^{\infty}}=W(0).
\end{equation}
Moreover, we see from \eqref{eq:1.9}, \eqref{eq:2.6} and \eqref{eq:2.4} that  
\begin{align}
\alpha_{\omega}\|\widetilde{\Phi}_{\omega}\|_{L^{2}}^{2}
+
\|\nabla \widetilde{\Phi}_{\omega}\|_{L^{2}}^{2}
-
\beta_{\omega}\|\widetilde{\Phi}_{\omega}\|_{L^{p+1}}^{p+1}
-
\|\widetilde{\Phi}_{\omega}\|_{L^{2^{*}}}^{2^{*}}
&=0, 
\nonumber
\end{align}and
\begin{align}
\label{eq:2.11}
\frac{\alpha_{\omega}}{d} 
\|\widetilde{\Phi}_{\omega} \|_{L^{2}}^{2}
-
\frac{2^{*}-(p+1)}{2^{*} (p+1)}
\beta_{\omega} 
\|\widetilde{\Phi}_{\omega} \|_{L^{p+1}}^{p+1}
&
=0
.
\end{align}

The following lemma tells us the asymptotic behavior of $M_{\omega}$ and $\alpha_\omega$ as $\omega \to \infty$: 
\begin{lemma}\label{proposition:2.3}
Assume $d\ge 3$, $1<p<\frac{d+2}{d-2}$ and $\Go \neq \emptyset$. Then 
\begin{align}
\label{eq:2.12}
\lim_{\omega \to \infty} \inf_{\Pho \in \Go } M_{\omega} &= \infty,
\\
\label{eq:2.13}
\lim_{\omega \to \infty} \sup_{\Pho \in \Go} \alpha_\omega &=0.
\end{align} 
\end{lemma}
\begin{proof}
First, we prove \eqref{eq:2.12}. Let $\omega>0$ and $\Pho \in \Go$. 
Since $0$ is a maximum point of $\Pho$ by Proposition \ref{proposition:1.2}, 
we see $\Delta \TPho (0) \leq 0$. 
Recalling $\TPho(0)=1$, we see from \eqref{eq:1.9} that
	\begin{equation}\label{eq:2.14}
		\omega - M_{\omega}^{p-1} - M_{\omega}^{\frac{4}{d-2}} \le 0,
	\end{equation}
which implies \eqref{eq:2.12}. 
\par 
Next, we prove \eqref{eq:2.13}. From \eqref{eq:2.12} and \eqref{eq:2.14}, we may assume $M_\omega \geq 1$ and $\omega \leq 2 M_\omega^{\frac{4}{d-2}}$. Furthermore, we see from the definition of $\alpha_\omega$ that 
	\[
		0 \leq \alpha := \limsup_{\omega \to \infty} \sup_{\Pho \in \Go} \alpha_\omega 
		\leq 2. 
	\]

	We prove \eqref{eq:2.13} by contradiction and suppose $0<\alpha\leq 2$. 
Then, we can take sequences $\{\omega_{n}\}$ and $\{\Phon \}$ 
such that $\lim_{n\to \infty}\omega_{n}=\infty$, $\Phon \in \Gon$ for each $n\ge 1$ and 
$\lim_{n\to \infty}\alpha_{\omega_{n}}=\alpha$. 
Since $\alpha>0$, we may also assume that 
	\[
		\inf_{n\ge 1} \alpha_{\omega_{n}} \ge \frac{\alpha}{2} .
	\]
Combining the above inequality 
with the definition of $\TPhon$, \eqref{eq:2.7} and \eqref{eq:2.8}, 
 we see that 
\begin{equation*}
\begin{split}
\alpha \| \TPhon \|_{L^{2}}^{2}
+
\|\nabla  \TPhon \|_{L^{2}}^{2}
&
\le
 2\alpha_{n} \| \TPhon \|_{L^{2}}^{2}
+
\|\nabla  \TPhon \|_{L^{2}}^{2}
\\[6pt]
&\lesssim 
\omega_{n}\| \Phi_{\omega_{n}} \|_{L^{2}}^{2}
+
\|\nabla \Phi_{\omega_{n}} \|_{L^{2}}^{2}
\lesssim
1+ 
\|\nabla W\|_{L^{2}}^{2}. 
\end{split} 
\end{equation*}
Hence, $\{\TPhon \}$ is bounded in $H^{1}(\mathbb{R}^{d})$. We also have $\|\TPhon\|_{L^{\infty}}=\TPhon(0)=1$ (see \eqref{eq:2.10}).

	Since $\TPhon$ satisfies equation \eqref{eq:1.9} with $\omega = \omega_n$ 
and $\beta_{\omega_n} \to 0$ as $n \to \infty$ 
due to \eqref{eq:2.12}, 
we find from the $W^{2,q}$ estimate and Schauder's estimate (see \cite{GT}) that 
there exists a subsequence of $\{\TPhon \}$ (still denoted by the same symbol) 
and a radial function $\widetilde{\Phi} \in H^{1}(\mathbb{R}^{d})$ such that 
	\begin{equation*}
		\lim_{n\to \infty}\TPhon = \widetilde{\Phi} \quad 
		\text{weakly in $H^1(\Rd)$ and strongly in 
		$C^2_{\rm loc}(\Rd)$}
\end{equation*}
and 
\begin{equation}\label{eq:2.15}
\left\{ \begin{array}{l}
-\Delta \widetilde{\Phi} + \alpha \widetilde{\Phi} 
= \widetilde{\Phi}^{\frac{d+2}{d-2}} \quad {\rm in} \ \Rd,
\\[6pt]
\widetilde{\Phi}(0)=1.
\end{array} \right. 
\end{equation}
On the other hand, Pohozaev's identity associated with the equation in \eqref{eq:2.15} implies that if $\alpha>0$, then  $\widetilde{\Phi} \equiv 0$ 
(see, e.g., \cite[Section 2.2]{Berestycki-Lions}). This is a contradiction. Thus, $\alpha=0$ and we have completed the proof of the lemma. 
\end{proof}
Now, we are ready to prove Proposition \ref{proposition:2.1}.

\begin{proof}[Proof of Proposition \ref{proposition:2.1}]
By contradiction, assume that there exist a constant $\varepsilon_{0}>0$,  
a sequence $\{ \omega_n \}$ in $(0,\infty)$ and 
a sequence $\{\TPhon \}$ such that 
$\lim_{n\to \infty}\omega_{n}=\infty$, $\TPhon \in \TGon$ and 
\begin{equation}\label{eq:2.16}
\lim_{n\to \infty}\| \TPhon - W \|_{\dH}
\ge \varepsilon_{0}.
\end{equation}
We remark that $\{\TPhon\}$ is bounded in $\dHRd$, 
$\|\TPhon\|_{L^{\infty}}=\TPhon(0)=1$, 
$\TPhon$ is a positive solution to \eqref{eq:1.9} with $\omega=\omega_{n}$ 
and $\lim_{n\to \infty}\alpha_{\omega_{n}}=\lim_{n\to \infty}\beta_{\omega_{n}}= 0$.  
Hence, as in the proof of Lemma \ref{proposition:2.3}, we can verify that 
there exist a subsequence of $\{\TPhon\}$ (still denoted by the same symbol) and a radial function $\widetilde{\Phi} \in \dot{H}^{1}(\mathbb{R}^{d})$ such that 
\begin{equation*}
\lim_{n\to \infty}\TPhon = \widetilde{\Phi} 
\quad 
\mbox{weakly in $\dHRd$ and strongly in $C^2_{\rm loc}(\Rd)$} 
\end{equation*}
and 
\begin{equation}\label{eq:2.17}
\left\{ \begin{array}{l}
\Delta \widetilde{\Phi}+\widetilde{\Phi}^{\frac{d+2}{d-2}}=0 \quad {\rm in} \ \Rd,
\\[6pt]
\widetilde{\Phi} (0) = 1
.
\end{array} \right.  
\end{equation}
From the uniqueness of radial solutions to the problem \eqref{eq:2.17} 
(see \cite{CGS}), it follows that $\widetilde{\Phi} = W$.

 On the other hand, we see from the weak lower semicontinuity of the $\dot{H}^{1}$-norm and Lemma \ref{proposition:2.2} that 
\begin{equation*}
\| W \|_{\dH} = \| \widetilde{\Phi} \|_{\dH} 
\le 
\liminf_{n\to \infty}\| \TPhon \|_{\dH}
\le \| W \|_{\dH}
\end{equation*}  
and therefore 
\begin{equation}\label{eq:2.18}
\lim_{n\to \infty}\| \TPhon \|_{\dH}= \| \widetilde{\Phi} \|_{\dH}=\|W \|_{\dH}
.
\end{equation}  
Combining the weak convergence in $\dot{H}^{1}(\mathbb{R}^{d})$, $\widetilde{\Phi}=W$ and \eqref{eq:2.18}, we find that 
\begin{equation*}
\lim_{n\to \infty} \TPhon = W
\quad 
\mbox{strongly in $\dHRd$}.
\end{equation*}  
However, this contradicts \eqref{eq:2.16} and we have completed the proof. 
\end{proof}

%%%%%%%%%%%%%%%%%%%%%%%%%%%%%%%%%%%%%%%%%%%%%%%%%%%%%%%%%%%%%%%%%%%%%%%%%%%%%%

\section{Uniform decay estimate}\label{section:3}

In this section, we discuss uniform decay properties of the rescaled ground states. 
In particular, we aim to derive the following crucial uniform decay estimate:    
\begin{proposition}\label{proposition:3.1}
Assume $d\ge 3$, $1<p<\frac{d+2}{d-2}$ and $\Go \neq \emptyset$. 
Then, there exist two constants $\omega_{dec}>0$ and  $C_{dec}>0$ such that for any $\omega >\omega_{dec}$ and any 
$x\in \mathbb{R}^{d}$, 
\begin{equation*}
\sup_{\TPho \in \TGo} \widetilde{\Phi}_{\omega}(x) \le C_{dec} \left(1 +|x| \right)^{-(d-2)}
. 
\end{equation*}
\end{proposition}  
A proof of Proposition \ref{proposition:3.1} will be given in Section \ref{subsection:3.3}.  First, we derive the following convergence result from Propositions \ref{proposition:2.1} and \ref{proposition:3.1}:
\begin{corollary}\label{theorem:3.1}
Assume $d\ge 3$, $1<p<\frac{d+2}{d-2}$ and $\Go \neq \emptyset$. 
Then, for any $q> \frac{d}{d-2}$, we have 
\begin{equation*}
\lim_{\omega \to \infty} \sup_{\TPho \in \TGo} \| \widetilde{\Phi}_{\omega} -W \|_{L^{q}}=0
.
\end{equation*}
\end{corollary}
\begin{proof}
Note first that Proposition \ref{proposition:2.1} together with Sobolev's inequality gives us
\begin{equation}\label{eq:3.1}
\lim_{\omega \to \infty}\sup_{\TPho \in \TGo} \| \TPho - W \|_{L^{2^\ast}} = 0
.
\end{equation}
Moreover, it follows from \eqref{eq:2.10} that  for each $\omega>0$ and $q > 2^\ast$, 
	\[
		\| \TPho - W \|_{L^q}^q \leq \int_{\Rd} 2^{q-2^\ast} 
		| \TPho - W |^{2^\ast} d x = 2^{q-2^\ast} \| \TPho - W \|_{L^{2^\ast}}^{2^\ast}.
	\]
Hence, we find that for any $q > 2^\ast $, 
\begin{equation*}
		\lim_{\omega \to \infty}
\sup_{\TPho \in \TGo} \| \TPho -W \|_{L^{q}} 
=0.
\end{equation*}

	Next, assume $\frac{d}{d-2}<q<2^{*}$ and fix $q_{0} \in (\frac{d}{d-2}, q)$. 
From Proposition \ref{proposition:3.1} we see that for any sufficiently large $\omega$, 
	\[ 
		\sup_{\TPho \in \TGo} \| \TPho \|_{L^{q_0}} \lesssim 1,
	\]
where the implicit constant depends on $q_{0}$. 
Furthermore, by H\"older's inequality and \eqref{eq:3.1}, we get 
	\begin{equation*}
		\lim_{\omega \to \infty} \sup_{\TPho \in \TGo} \| \TPho - W \|_{L^{q}} 
\lesssim
\lim_{\omega \to \infty} \sup_{\TPho \in \TGo} 
\| \TPho - W \|_{L^{2^{*}}}^{\frac{2^{*}(q-q_{0})}{q(2^{*}-q)}} 
= 0.
	\end{equation*}
Thus, we have completed the proof of Corollary \ref{theorem:3.1}. 
\end{proof}

%%%%%%%%%%%%%%%%%%%%%%%%%%%%%%%%%%%%%%%%%%%%%%%%%%%%%%%%%%%%%%%%%%%%%%%%%%%%%%%%

\subsection{Exponential decay estimate}\label{subsection:3.1}

In this subsection, we derive an exponential decay estimate which we need 
 in the proof of Proposition \ref{proposition:3.1}. Let us begin with rephrasing  the estimate \eqref{eq:2.9} 
 in terms of the rescaled ground state: 
 for every $\omega>0$, $\TPho \in \TGo$ and $x \in \mathbb{R}^{d}\setminus \{0\}$, 
\begin{equation}\label{eq:3.2}
\begin{split}
\widetilde{\Phi}_{\omega}(x)
&=
M_{\omega}^{-1}\Phi_{\omega}(M_{\omega}^{-\frac{2}{d-2}}x)
\\[6pt]
&\lesssim
M_{\omega}^{-1} \omega^{-\frac{1}{4}}( M_{\omega}^{-\frac{2}{d-2}}|x|)^{-\frac{d-1}{2}}
= \alpha_\omega^{-\frac{1}{4}}|x|^{-\frac{d-1}{2}}.
\end{split}
\end{equation}

Next, we state the main result in this subsection:  
\begin{lemma}\label{proposition:3.2}
Assume $d\ge 3$, $1<p<\frac{d+2}{d-2}$ and $\Go \neq \emptyset$. 
Then, there exist constants $L_{0}>0$ and $C_0>0$ such that for any $ \omega >0$ and $|x| \geq L_{0} \alpha_\omega^{-\frac{1}{2}}$, 
\begin{equation}\label{eq:3.3}
\sup_{\TPho \in \TGo} \widetilde{\Phi}_{\omega}(x)
\leq 
C_{0} \alpha_\omega^{\frac{d-2}{4}} e^{-\frac{\sqrt{\alpha_\omega}}{2}|x|}.
\end{equation}
\end{lemma}
\begin{proof}
Let $\omega>0$, $\TPho \in \TGo $ and $L_{0}>0$ be a large number to be specified later. 
Since $\TPho$ is strictly decreasing in the radial direction by Proposition \ref{proposition:1.2}, 
we see from \eqref{eq:3.2} that if $|x| \geq L_{0} \alpha_{\omega}^{-\frac{1}{2}}$, then 
\begin{equation}\label{eq:3.4}
\TPho(x) \le \TPho ( L_{0} \alpha_\omega^{-\frac{1}{2}} ) 
\lesssim L_{0}^{-\frac{d-1}{2}} 
		\alpha_\omega^{\frac{d-2}{4}}.
\end{equation}

	Next, we rewrite \eqref{eq:1.9} as
	\begin{equation}\label{eq:3.5}
		-\Delta \TPho + \Big( \alpha_\omega - \beta_\omega \TPho^{p-1} 
		- \TPho^{\frac{4}{d-2}} \Big) \TPho = 0 \quad 
		{\rm in} \ \Rd. 
	\end{equation}
We see from \eqref{eq:3.4} and the definitions of $\alpha_{\omega}$ and $\beta_{\omega}$ that  if $|x| \geq L_{0} \alpha_{\omega}^{-\frac{1}{2}}$, then 
	\begin{align} 
\label{eq:3.6}
\beta_{\omega} \TPho^{p-1} (x) 
&\lesssim 
L_{0}^{-\frac{(d-1)(p-1)}{2}}
\alpha_{\omega}^{\frac{(d-2)(p-1)}{4}} 
\beta_{\omega} 
= 
L_{0}^{-\frac{(d-1)(p-1)}{2}} \omega^{-\frac{4-(d-2)(p-1)}{4}}
\alpha_{\omega} 
,
\\[6pt]		
\label{eq:3.7}
\TPho^{\frac{4}{d-2}}(x) 
&\lesssim L_{0}^{-\frac{2(d-1)}{d-2}} \alpha_{\omega}.  
\end{align}
Furthermore, it follows from \eqref{eq:3.5}, \eqref{eq:3.6}, \eqref{eq:3.7} and the assumption $p<\frac{d+2}{d-2}=1+\frac{4}{d-2}$ that 
if we choose a sufficiently large $L_0$ depending only on $d$ and $p$, then 
	\begin{equation}\label{eq:3.8}
		-\Delta \TPho(x) + \frac{\alpha_\omega}{2} \TPho(x) \leq 0  
	\end{equation}
for all $|x| \geq L_{0} \alpha_\omega^{-\frac{1}{2}}$. 
\par 
Now, we shall derive \eqref{eq:3.3} by using the comparison  principle. 
To this end, let $R> L_{0} \alpha_\omega^{-\frac{1}{2}}$ and introduce a positive function $\psi_{R}$ on $\mathbb{R}$ as 
	\begin{equation*}
		\psi_R(r) := \exp \Big( - \frac{\alpha_{\omega}^{\frac{1}{2}}}{2} \big( r - L_{0} \alpha_\omega^{-\frac{1}{2}} \big) \Big) 
		+ 
\exp \Big( \frac{\alpha_{\omega}^{\frac{1}{2}}}{2}  ( r - R ) \Big).
	\end{equation*}
It is easy to verify that  
	\begin{equation}\label{eq:3.9}
		\left|\psi_{R}'(r)\right| \leq \frac{\alpha_{\omega}^{\frac{1}{2}}}{2} \psi_R(r), 
		\quad \psi_{R}''(r) = \frac{\alpha_\omega}{4} \psi_R(r). 
	\end{equation}
We use the same symbol $\psi_{R}$ to denote the radial function $\psi_{R}(|x|)$  on $\mathbb{R}^{d}$. 
Then, we see from \eqref{eq:3.9} that if $L_{0} \ge  2(d-1)$ and $L_{0}\alpha_\omega^{-\frac{1}{2}} < r < R$, then 
	\begin{equation}\label{eq:3.10}
-\Delta \psi_{R}+\frac{\alpha_{\omega}}{2}\psi_{R}
=
-\psi_{R}''-\frac{d-1}{r}\psi_{R}'+\frac{\alpha_{\omega}}{2}\psi_{R} 
\geq 0 .
	\end{equation}
Furthermore, it follows from \eqref{eq:3.4}, $\psi_R( L_{0} \alpha_\omega^{-\frac{1}{2}} ) \geq 1$ and $\psi_R(R)\ge 1$ that 
	\begin{equation}\label{eq:3.11}
		\TPho (R) \le \TPho (L_{0}\alpha_{\omega}^{-\frac{1}{2}})  
		\lesssim 
L_{0}^{-\frac{d-1}{2}} \alpha_\omega^{\frac{d-2}{4}} 
		\min 
\big\{ 
\psi_R( L_{0}\alpha^{-\frac{1}{2}} ), \ \psi_R(R)  
\big\}. 
	\end{equation}
Hence, the comparison principle together with \eqref{eq:3.8}, \eqref{eq:3.10} and \eqref{eq:3.11} implies that 
if $L_{0}\alpha_\omega^{-\frac{1}{2}} \leq |x| \leq R$, then  
	\[
		\TPho (x) \lesssim 
L_{0}^{-\frac{d-1}{2}} \alpha_{\omega}^{\frac{d-2}{4}} \psi_R(|x|).
	\]
Since $R > L_{0}\alpha_{\omega}^{-\frac{1}{2}}$ is arbitrary, 
taking $R \to \infty$, we find that  
	\[
					\TPho (x) 
			\lesssim  L_{0}^{-\frac{d-1}{2}} \alpha_\omega^{\frac{d-2}{4}} 
			\exp \Big( - \frac{\alpha_{\omega}^{\frac{1}{2}}}{2} 
				\big( |x| - L_{0} \alpha_\omega^{-\frac{1}{2}} 
\big) 
\Big) 
			= L_{0}^{-\frac{d-1}{2}} e^{\frac{L_{0}}{2}} 
\alpha_\omega^{\frac{d-2}{4}}  
e^{- \frac{ \sqrt{\alpha_{\omega}}}{2} |x|}
	\]
for all $|x| \geq L_{0} \alpha_\omega^{-\frac{1}{2}}$, which is the desired estimate \eqref{eq:3.3}. 
\end{proof}

%%%%%%%%%%%%%%%%%%%%%%%%%%%%%%%%%%%%%%%%%%%%%%%%%%%%%%%%%%%%%%%%%%%%%%%%%%%%%%%%
\subsection{Kelvin transforms of rescaled ground states}\label{subsection:3.2}

In this subsection, we consider the Kelvin transform of elements in $\TGo$. 
We use $K[u]$ to denote the Kelvin transform of a function $u$, that is, 
\begin{equation*}
K[u](x):= |x|^{-(d-2)} u \left(\frac{x}{|x|^{2}}\right)
.
\end{equation*}
Remark that $\|K[\TPho]\|_{L^{\infty}(B_{1})} \lesssim 1$ implies
\begin{equation*}
\sup_{|x|\ge 1}\TPho(x) \lesssim |x|^{-(d-2)}
.
\end{equation*}
Thus, to prove Proposition \ref{proposition:3.1}, 
it suffices to show that there exists $\omega_{dec}>0$ such that 
\begin{equation}\label{eq:3.12}
\sup_{\omega>\omega_{dec}} \sup_{\TPho \in \TGo} \| K[\TPho]\|_{L^{\infty}(B_{1})} <\infty.
\end{equation}

It is easy to verify that $K[\TPho]$ satisfies 
\begin{equation}\label{eq:3.13}
-\Delta K[\TPho]
+ 
\alpha_{\omega}\frac{1}{|x|^{4}}
K[\TPho]
=
\beta_{\omega}\frac{1}{|x|^{\gamma}}
K[\TPho]^{p}
+
K[\TPho]^{\frac{d+2}{d-2}}
,
\end{equation}
where 
\begin{equation}\label{eq:3.14}
\gamma:=  4 - (d-2)(p-1) > 0.
\end{equation}  
We also see from Lemma \ref{proposition:3.2} that if $|x|\le \alpha_{\omega}^{\frac{1}{2}}  /L_{0}$, then 
\begin{equation}\label{eq:3.15}
K[\TPho](x) \lesssim \alpha_{\omega}^{\frac{d-2}{4}} 
|x|^{-(d-2)} e^{-\frac{\sqrt{\alpha_{\omega}}}{2} |x|^{-1}}
. 
\end{equation}
Furthermore, since the Kelvin transform is linear and preserves the $\dHRd$ norm, we have 
\begin{equation}\label{eq:3.16}
\| K[u] - K[v] \|_{\dH} = \| u - v \|_{\dH}
\end{equation}  
for any $u,v\in \dot{H}^{1}(\mathbb{R}^{d})$. Hence, Proposition \ref{proposition:2.1} leads us to the following result:

\begin{lemma}\label{proposition:3.3}
Assume $d\ge 3$, $1<p<\frac{d+2}{d-2}$ and $\Go \neq \emptyset$. Then, it holds that 
\begin{equation*}
\lim_{\omega\to \infty}\sup_{\TPho \in \TGo} 
\| K [\TPho] - K[W] \|_{\dot{H}^{1}}=0.
\end{equation*}
\end{lemma}
In order to use Moser's iteration 
(see Proposition \ref{proposition:B.1}) for \eqref{eq:3.13}, 
we need the following lemma:

\begin{lemma}\label{proposition:3.4}
Assume $d\ge 3$, $1<p<\frac{d+2}{d-2}$ and $\Go \neq \emptyset$. Then, it holds that 
\begin{equation}\label{eq:3.17}
\lim_{\omega \to \infty} \sup_{\TPho \in \TGo} \int_{|x|\le 4}
\Big| \frac{\beta_{\omega}}{|x|^{\gamma}}
K[\TPho]^{p-1}(x)
\Big|^{\frac{d}{2}}dx
=
0.
\end{equation}
\end{lemma}

\begin{proof}
It follows from Lemma \ref{proposition:2.3} that there exists $\omega_{1}>0$ such that 
$ \alpha_\omega \le 1$ for all $\omega \ge \omega_{1}$ and $\TPho \in \TGo$. 
In what follows, we always assume that $\omega>\omega_{1}$. 
Also let $L_0 \geq 1$ be the constant appeared in Lemma \ref{proposition:3.2}. 
We divide the integral into two parts:
\begin{align*}
I_{\omega,in}
&:=
\int_{|x|\le \sqrt{\alpha_\omega}/L_0}
\left| \frac{\beta_{\omega}}{|x|^{\gamma}}
K[\TPho]^{p-1}(x)
\right|^{\frac{d}{2}}
dx,
\\[6pt] 
I_{\omega,out}
&:=
\int_{ \sqrt{\alpha_\omega}/L_0 \le |x|\le 4}
\left| \frac{\beta_{\omega}}{|x|^{\gamma}}
K[\TPho]^{p-1}(x)
\right|^{\frac{d}{2}}
dx.
\end{align*}

	We first show 
	\begin{equation}\label{eq:3.18}
	\lim_{\omega \to \infty} \sup_{\TPho \in \TGo}  I_{\omega,out}=0.
	\end{equation}
To this end, set $s_{0}:=\frac{4}{(d-2)(p-1)}$ and note that thanks to 
$0 < p-1 < \frac{4}{d-2}$, we have $1 < s_0 < \infty$. 
Moreover, since $\frac{d(p-1)s_{0}}{2}=2^{*}$ and $1-\frac{1}{s_{0}}=\frac{\gamma}{4}$, we see from H\"older's inequality and Lemma \ref{proposition:3.3} that   
\[
\begin{aligned}
I_{\omega,out}
&= \beta_\omega^{\frac{d}{2}} 
\int_{ \sqrt{\alpha_\omega} / L_0 \le |x|\le 4} 
|x|^{- \frac{d}{2} \gamma  } K[\TPho]^{ \frac{d(p-1)}{2} }(x)\,dx 
\\[6pt]
&\le 
\beta_\omega^{\frac{d}{2}}
\left\{
\int_{ \sqrt{\alpha_\omega} / L_0 \le |x|\le 4}
K[\TPho]^{2^{*}}(x)\,dx 
\right\}^{1/s_0} 
\left\{
\int_{ \sqrt{\alpha_\omega} / L_0 \le |x|\le 4}
|x|^{-2d} \,dx 
\right\}^{\frac{\gamma}{4}}
\\[6pt]
&\lesssim 
\beta_\omega^{\frac{d}{2}}
\left\{
\int_{\sqrt{\alpha_\omega} / L_0}^{4} 
r^{-d-1} \,dr 
\right\}^{\frac{\gamma}{4}}
\lesssim
\beta_\omega^{\frac{d}{2}} 
\alpha_\omega^{ -\frac{d\gamma}{8} }
=
\omega^{-\frac{d\gamma}{8}} . 
\end{aligned}
\]
Thus, \eqref{eq:3.18} holds.

	Next, we consider $I_{\omega,in}$. We see from \eqref{eq:3.15} and $\gamma + (d-2)(p-1)= 4$ (see \eqref{eq:3.14}) that 
\begin{equation}\label{eq:3.19}
\begin{aligned}
I_{\omega,in}
&= \beta_\omega^{\frac{d}{2}} 
\int_{|x| \le  \sqrt{\alpha_\omega}/ L_0 }
|x|^{-\frac{d\gamma}{2}} K[\TPho]^{\frac{d}{2}(p-1)}(x)\,dx
\\[6pt]
&\lesssim 
\beta_\omega^{\frac{d}{2}} 
\alpha_\omega^{ \frac{d(d-2)(p-1)}{8} } 
\int_{|x| \le \sqrt{\alpha_\omega}/ L_0} 
|x|^{- \frac{d}{2} \{ \gamma +(d-2)(p-1)\} } 
\exp\left( - \frac{d(p-1) \sqrt{\alpha_\omega}}{ 4|x|} \right) dx
\\[6pt]
&\lesssim
\beta_\omega^{\frac{d}{2}} 
\alpha_\omega^{ \frac{d(d-2)(p-1)}{8}} 
\int_0^{\sqrt{\alpha_\omega}/L_0}
r^{-d-1} 
\exp\left( \frac{-d(p-1)\sqrt{\alpha_\omega}}{4r}\right) \,dr 
.
\end{aligned}
\end{equation}
Furthermore, by the change of variables $s = \sqrt{\alpha_\omega} r^{-1}$, 
we find from \eqref{eq:3.19} that  
	\begin{equation*}
		I_{\omega,in} \lesssim 
		\beta_\omega^{\frac{d}{2}} 
		\alpha_\omega^{ \frac{d(d-2)(p-1)}{8} - \frac{d}{2} } 
		\int_{L_0}^\infty s^{d-1} \exp\left( -\frac{d(p-1) s}{4}\right) ds
.
        \end{equation*}
Since $\frac{d}{8} \gamma = \frac{d}{2} - \frac{d(d-2)(p-1)}{8}$, we see that 
\begin{equation*}
\beta_\omega^{\frac{d}{2}} 
		\alpha_{\omega}^{ \frac{d(d-2)(p-1)}{8} - \frac{d}{2} } 
= \beta_\omega^{\frac{d}{2}} \alpha_\omega^{-\frac{d}{8}\gamma} 
= \omega^{-\frac{d\gamma}{8}}.
\end{equation*}
Hence, we conclude that  
\begin{equation}\label{eq:3.20}
\lim_{\omega\to \infty} \sup_{\TPho \in \TGo} I_{\omega, in}=0
.
\end{equation} 
By \eqref{eq:3.18} and \eqref{eq:3.20}, we obtain the desired result \eqref{eq:3.17}. 
\end{proof}

%%%%%%%%%%%%%%%%%%%%%%%%%%%%%%%%%%%%%%%%%%%%%%%%%%%%%%%%%%%%%%%%%%%%%%%%%%%%%%%

\subsection{Proof of Proposition \ref{proposition:3.1}}\label{subsection:3.3}

Now we prove Proposition \ref{proposition:3.1}:

\begin{proof}[Proof of Proposition \ref{proposition:3.1}]
As mentioned in \eqref{eq:3.12}, it suffices to show 
	\[
	\sup_{\omega>\omega_{dec}} \sup_{\TPho \in \TGo} \| K[\TPho]\|_{L^{\infty}(B_{1})} <\infty.
	\]
Recall that $K[\TPho]$ is a solution to \eqref{eq:3.13}. 
We shall prove \eqref{eq:3.12} by applying Proposition \ref{proposition:B.1} to $K[\TPho] $ 
with 
\begin{equation}\label{eq:3.21}
a(x) = \frac{\alpha_\omega}{|x|^4}, \quad 
b(x) 
= 
\frac{\beta_\omega}{|x|^\gamma} K[\TPho]^{p-1}(x) 
 + K[\TPho]^{\frac{4}{d-2}} (x) .
\end{equation}
First, we note that it follows from \eqref{eq:3.15} that for any $v \in H_{0}^{1}(B_{4})$, 
\begin{equation}\label{eq:3.22}
\int_{B_{4}} \frac{\alpha_\omega}{|x|^4} K[\TPho](x) |v(x)|\,dx 
<\infty. 
\end{equation}
By \eqref{eq:3.16}, 
Proposition \ref{proposition:2.1} and Sobolev's inequality, one has
	\[
	\lim_{\omega \to \infty} \sup_{\TPho \in \TGo} 
			\left\| K[\TPho] - K[W] \right\|_{L^{2^\ast}(\Rd)} = 0,
	\]
implying that the family $\{ K [ \TPho ]^{2^\ast} \}$ is uniformly 
integrable. Hence, it is not difficult to check that 
	\begin{equation}\label{eq:3.23}
		\lim_{\omega \to \infty} \sup_{\TPho \in \TGo} 
		\left\| K[\TPho]^{\frac{4}{d-2}} - K[W]^\frac{4}{d-2} \right\|_{L^{\frac{d}{2}}(\Rd)} = 0.
	\end{equation}
Combining \eqref{eq:3.22} and \eqref{eq:3.23} with Lemma \ref{proposition:3.4}, 
we may apply Proposition \ref{proposition:B.1} (i) 
to show that for every $q>1$ there exists $\omega_{dec,q} > 0$ such that 
	\begin{equation}\label{eq:3.24}
		\sup_{\omega > \omega_{dec,q}} \sup_{\TPho \in \TGo } 
		\left\| K[\TPho]^{q} \right\|_{H^1(B_1)} \leq C_{q}. 
	\end{equation}

	Next, let $\eta$ be a non-increasing smooth function on $[0,\infty)$ 
such that $\eta(r)=1$ for $r \leq 1/2$ and $\eta(r)=0$ for $r\ge 1$. 
It is easily seen that for any $x \in \mathbb{R}^{d}\setminus \{0\}$ and any $q>1$, 
\begin{equation}\label{eq:3.25} 
K \big[\eta K[\TPho]^{q} \big] (x)
=
|x|^{-(d-2)} \eta \left( \frac{1}{|x|} \right)
K[\TPho]^{q} \left(\frac{x}{|x|^{2}} \right) 
=
|x|^{(d-2)(q-1)} 
\eta \left( \frac{1}{|x|} \right)\TPho^q (x).
\end{equation}
It follows from \eqref{eq:3.16}, \eqref{eq:3.24} and \eqref{eq:3.25} that for any $q>1$, 
there exists  $C_{q}>0$ such that 
\begin{equation}\label{eq:3.26} 
\sup_{\omega > \omega_{dec,q}} \sup_{\TPho \in \TGo }
\left\|  |x|^{(d-2)(q-1)}\eta \left(\frac{1}{|x|}\right) \TPho^{q}  \right\|_{\dH} 
= \sup_{\omega > \omega_{dec,q}} \sup_{\TPho \in \TGo }\left\| \eta  K[\TPho]^{q} \right\|_{\dH}
\le C_{q}
.
\end{equation}
Furthermore, by \cite[Lemma A.III]{Berestycki-Lions} and \eqref{eq:3.26}, it holds that 
	\[
	\sup_{\omega > \omega_{dec,q}} \sup_{\TPho \in \TGo } 
	\left( |x|^{(d-2)(q-1)} \TPho^{q}(x) \right) 
 \lesssim C_{q} |x|^{-\frac{d-2}{2}} 
	\] 
for all $|x|\ge 2$. This implies that for any $q>1$ and any $|x|\le 1/2$, 
	\begin{equation}\label{eq:3.27}
		\sup_{\omega > \omega_{dec,q}} \sup_{\TPho \in \TGo } K[\TPho](x) \lesssim 
		C_{q}|x|^{-\frac{d-2}{2q}}.
	\end{equation}

	To prove \eqref{eq:3.12}, we shall apply Proposition \ref{proposition:B.1} (ii). 
Since we have \eqref{eq:3.21}, \eqref{eq:3.22} and \eqref{eq:3.24}, what remains to prove is that there exist $q_0>d/2$ and 
$\omega_{dec} > 0$ such that 
	\begin{equation}\label{eq:3.28}
		\sup_{\omega > \omega_{dec}} \sup_{\TPho \in \TGo } 
		\int_{|x| \leq 4}
		\left| \frac{\beta_\omega}{|x|^\gamma} K[\TPho]^{p-1} \right|^{q_0}  dx 
		\le C_{q_0}.
	\end{equation}

To this end, 
we divide the proof of \eqref{eq:3.28} into two cases. 

\medskip

\noindent 
{\bf Case 1:} $\frac{d}{d-2}<p<\frac{d+2}{d-2}$.

\medskip

	We first remark that the condition 
$\frac{d}{d-2} < p < \frac{d+2}{d-2}$ implies $ 0 < \gamma < 2 $. 
Therefore, we may choose $q_0>d/2$ and sufficiently large $q > 1$ such that 
	\[
		\left[ \gamma + \frac{(d-2)(p-1)}{2q} \right] q_0 < d. 
	\]
It follows from \eqref{eq:3.27} that 
	\[
		\sup_{\omega > \omega_{dec,q}} \sup_{\TPho \in \TGo } 
		\left(\frac{1}{|x|^\gamma} K[ \TPho ]^{p-1} (x)\right)^{q_0} 
		\lesssim  C_q |x|^{ - \gamma q_0 - \frac{(d-2)(p-1)}{2q} q_0 } \in L^1(B_{1/2}). 
	\]
From this, \eqref{eq:3.28} holds.

\medskip

\noindent
{\bf Case 2:} $1<p\le \frac{d}{d-2}$.

\medskip

We remark that $1 < p \leq \frac{d}{d-2}$ gives $ 2 \leq \gamma < 4$ and 
$\frac{d}{2} < \frac{2d}{\gamma}$. 
Let $q_0 \in ( \frac{d}{2} , \frac{2d}{\gamma}  )$ 
and we claim that \eqref{eq:3.28} holds for this $q_0$. 
For this purpose, we remark that 
by $1 - \frac{d(p-1)}{2} < 1$, there exist $\theta > \frac{1}{2}$ and $s_1 > 1$ so that 
	\begin{equation}\label{eq:3.29}
		\frac{4d \theta }{4\theta -1}>\gamma q_{0}, \quad 
		\frac{4d \theta }{(4\theta -1)\gamma q_{0}} \ge s_{1}>1, \quad 
		\frac{d(p-1)}{2} s_1 > s_1 - 1. 
	\end{equation}
Next, we divide the integral into two parts:
\begin{align*}
J_{\omega,out}
&:=
\int_{ \alpha_{\omega}^{\theta} \le |x|\le 4}
\Big| \frac{\beta_{\omega}}{|x|^{\gamma}}
K[\TPho]^{p-1}
\Big|^{q_{0}}
dx ,
\\[6pt]
J_{\omega,in}
&:=
\int_{|x|< \alpha_{\omega}^{\theta}}
\Big| \frac{\beta_{\omega}}{|x|^{\gamma}}
K[\TPho]^{p-1}
\Big|^{q_{0}}
dx .
\end{align*}

	We first consider $J_{\omega,out}$. 
Since $\frac{(p-1)q_{0}s_{1}}{s_{1}-1}>1$ holds due to $q_0 > \frac{d}{2}$ and \eqref{eq:3.29}, 
by H\"older's inequality and \eqref{eq:3.24}, we see that  
\begin{equation*}
\begin{split}
J_{\omega,out}
&= \beta_{\omega}^{q_{0}} 
\int_{ \alpha_{\omega}^{\theta} \le |x|\le 4} 
|x|^{- \gamma q_{0} } 
K[\TPho]^{(p-1)q_{0}}\,dx 
\\[6pt]
&\le 
\beta_{\omega}^{q_{0}}
\left\{
\int_{|x|\le 4}
K[\TPho]^{\frac{(p-1) q_{0} s_{1}}{s_{1}-1}}
\,dx 
\right\}^{1-\frac{1}{s_{1}}} 
\left\{
\int_{ \alpha_{\omega}^{\theta} \le |x|\le 4}
|x|^{- \gamma q_{0} s_{1}} \,dx 
\right\}^{\frac{1}{s_{1}}}
\\[6pt]
&\le C(q_{0},s_{1}) 
\beta_\omega^{q_{0}}
\left\{
\int_{\alpha_{\omega}^{\theta}}^{4} 
r^{-\gamma q_{0} s_{1}+d-1} \,dr 
\right\}^{\frac{1}{s_{1}}}
\end{split}
\end{equation*} 
where $C(q_{0}, s_{1})>0$ depends only on $d$, $p$, $q_{0}$ and $s_{1}$. 
Recalling the definitions of $\alpha_{\omega}$, $\beta_{\omega}$ and $\gamma$, 
we obtain 
\begin{equation}\label{eq:3.30}
\begin{aligned}
J_{\omega,out} 
&\le  
C(q_{0}, s_{1})\beta_{\omega}^{q_{0}} 
\alpha_{\omega}^{\frac{(-\gamma q_{0} s_{1}+d)\theta}{ s_{1}}}
=
C(q_{0}, s_{1})
M_{\omega}^{-\frac{\gamma q_{0}}{d-2}}
\omega^{\frac{(-\gamma q_{0}s_{1} + d)\theta}{s_{1}}}
M_{\omega}^{\frac{4\theta \gamma q_{0} s_{1} - 4 d\theta}{(d-2)s_{1}}}
\\[6pt]
&=
C(q_{0}, s_{1})
\omega^{\frac{(-\gamma q_{0}s_{1} + d)\theta}{s_{1}}}
M_{\omega}^{-\frac{(4\theta -1) \gamma q_{0}}{(d-2) s_{1}}
\big\{ \frac{4d\theta}{(4\theta -1)\gamma q_{0}} - s_{1} \big\}}
.
\end{aligned}
\end{equation}
Since $2 \leq \gamma < 4$ and $s_1 > 1$ imply $-\gamma q_0 s_1 + d < 0$, 
\eqref{eq:3.30} and \eqref{eq:3.29} yield
\begin{equation}\label{eq:3.31}
\lim_{\omega \to \infty} \sup_{\TPho \in \TGo}  J_{\omega,out}=0.
\end{equation}

	Next, we consider $J_{\omega,in}$. 
Since $\theta > \frac{1}{2}$ and $\lim_{\omega \to \infty}\alpha_\omega = 0$, 
we may assume $\alpha^{-\theta}_\omega \leq L_0\alpha_\omega^{-\frac{1}{2}}$. 
Thus, from Lemma \ref{proposition:3.2} (or \eqref{eq:3.15}) 
and the definition of $\gamma$, it follows that 
\begin{equation}\label{eq:3.32}
\begin{aligned}
J_{\omega,in}
&= 
\beta_{\omega}^{q_{0}} 
\int_{|x| \le  \alpha_{\omega}^{\theta}  }
|x|^{-\gamma q_{0}} 
K[\TPho]^{(p-1)q_{0}} \,dx
\\[6pt]
&\lesssim 
\beta_{\omega}^{q_{0}} 
\alpha_\omega^{ \frac{(d-2)(p-1)q_{0}}{4} } 
\int_{|x| \le \alpha_{\omega}^{\theta}} 
|x|^{-4 q_{0}} 
\exp\left( - \frac{(p-1)q_{0} \sqrt{\alpha_\omega} }{ 2 |x|} \right) dx
\\[6pt]
&\lesssim 
\beta_{\omega}^{q_{0}} 
\alpha_\omega^{ \frac{(d-2)(p-1)q_{0}}{4} } 
\int_{0}^{\alpha_{\omega}^{\theta}} 
s^{- 4q_{0}+d-1} 
\exp\left( - \frac{(p-1)q_{0} \sqrt{\alpha_\omega}}{2 s} \right) \,ds 
.
\end{aligned}
\end{equation}
Using the change of variables $t = \alpha_{\omega}^{\theta} s^{-1}$, 
we find from  $\theta>1/2$ that 
\begin{equation}\label{eq:3.33}
\begin{split}
J_{\omega,in} 
&\lesssim 
\beta_{\omega}^{q_{0}} 
\alpha_{\omega}^{\frac{(d-2)(p-1)q_{0}}{4}-(4q_{0}-d)\theta} 
\int_{1}^{\infty} 
t^{4q_{0}-d-1} 
\exp\left( -\frac{(p-1)q_{0} \alpha_{\omega}^{-(\theta-1/2)}}{2} t\right)\, dt
\\[6pt]
&\le
C(q_{0}, \theta) 
\beta_{\omega}^{q_{0}} 
\alpha_{\omega}^{\frac{(d-2)(p-1)q_{0}}{4}-(4q_{0}-d)\theta} 
\int_{1}^{\infty} 
\exp\left( -\frac{(p-1)q_{0} \alpha_{\omega}^{-(\theta-1/2)}}{4} t\right)\, dt
\\[6pt]
&\le 
C(q_{0}, \theta) 
\beta_{\omega}^{q_{0}} 
\alpha_{\omega}^{\frac{(d-2)(p-1)q_{0}}{4}-(4q_{0}-d)\theta +\theta -\frac{1}{2}} 
\exp\left( -\frac{(p-1)q_{0} \alpha_{\omega}^{-(\theta-1/2)}}{4}\right)
\\[6pt]
&\le C(q_{0}, \theta)  \beta_{\omega}^{q_{0}}
\end{split}
\end{equation}
where $C(q_{0}, \theta)$ denotes a positive constant depending only on 
$d$, $p$, $q_{0}$ and $\theta$. Thus, it follows from \eqref{eq:3.33} 
and  $\lim_{\omega\to \infty}\beta_{\omega}=0$ that 
\begin{equation}\label{eq:3.34}
\lim_{\omega \to \infty} \sup_{\TPho \in \TGo}  J_{\omega,in}=0.
\end{equation}
By \eqref{eq:3.31} and \eqref{eq:3.34}, we have \eqref{eq:3.28} and complete the proof. 
\end{proof}

%%%%%%%%%%%%%%%%%%%%%%%%%%%%%%%%%%%%%%%%%%%%%%%%%%%%%%%%%%%%%%%%%%%%%%%%%%%%%%%%
\section{Proof of Theorem \ref{theorem:1.1}}\label{section:4}

In this section, we shall give a proof of Theorem \ref{theorem:1.1}. 
In the sequel, we assume $d \geq 5$, hence, 
by Proposition \ref{proposition:1.1}, we have $\Go \neq \emptyset$ 
and the results of Sections \ref{section:2} and \ref{section:3} hold.

Our proof is based on the ideas in \cite{Grossi, GLP} and 
we first prove the uniqueness by contradiction. 
Therefore, we suppose to the contrary that there exists  a sequence $\{\omega_{n}\}$ in $(0,\infty)$ such that 
$\lim_{n\to \infty}\omega_{n}=\infty$ and 
for each $n\ge 1$, $\Phi_{n,1}, \Phi_{n,2}\in \mathcal{G}_{\omega_{n}}$ and $\Phi_{n,1}\neq \Phi_{n,2}$. For $j=1,2$, we set 

\begin{equation}\label{eq:4.1}
\begin{aligned}
M_{n,j}&:=\max_{\mathbb{R}^{d}}\Phi_{n,j}=\Phi_{n,j}(0), 
\\[6pt] 
\widetilde{\Phi}_{n,j}(x)&:=M_{n,j}^{-1}\Phi_{n,j}(M_{n,j}^{-\frac{2}{d-2}}x),
\\[6pt]
\alpha_{n,j} &:= \omega_n M_{n,j}^{-\frac{4}{d-2}},  
\\[6pt] 
			\beta_{n,j} &:= M_{n,j}^{p-1-\frac{4}{d-2}},  
			\end{aligned}
\end{equation}
and
\begin{equation}%\\[6pt] 
			\mu_{n}:=\frac{M_{n,2}}{M_{n,1}}.
\end{equation}
We shall derive a contradiction. Let us begin with the following lemma: 
\begin{lemma}\label{proposition:4.1}
Assume $d\ge 5$ and $1<p<\frac{d+2}{d-2}$. Then, for $j=1,2$, 
\begin{equation}\label{eq:4.2}
\lim_{n \to \infty} \frac{\beta_{n,j}}{\alpha_{n,j}}
= 
\frac{2(p+1)}{4-(d-2)(p-1)}
\frac{\| W\|^2_{L^{2}}}{\|W\|_{L^{p+1}}^{p+1}}
\end{equation}
where $\alpha_{n,j}, \beta_{n,j}$ are given in \eqref{eq:4.1}, and $W$ in \eqref{eq:1.10}. 
Furthermore, it holds that  
\begin{equation}\label{eq:4.3}
\lim_{n\to \infty} \mu_{n} = 1 .
\end{equation}
\end{lemma}
\begin{proof}
The claim \eqref{eq:4.2} follows from identity \eqref{eq:2.11} and Corollary \ref{theorem:3.1}. Furthermore, an elementary computation together with \eqref{eq:4.2} shows \eqref{eq:4.3}.
\end{proof}
In what follows, owing to Lemma \ref{proposition:4.1}, we may assume the that for all $n$ and $j=1,2$: 
\begin{align}
\label{eq:4.4}
&\beta_{n,j} \lesssim \alpha_{n,j}, 
\\[6pt]
\label{eq:4.5}
&\frac{1}{2}\le \mu_{n} \le \frac{3}{2}. 
\end{align}
Moreover, Proposition \ref{proposition:2.1}, Corollary \ref{theorem:3.1} and Lemma \ref{proposition:4.1} imply that 
\begin{equation}\label{eq:4.6}
\lim_{n\to \infty}
\| \TPhnone -W \|_{H^{1}}
=
\lim_{n\to \infty}
\| \mu_{n} \TPhntwo( \mu_{n}^{\frac{2}{d-2}}\cdot )
-W
\|_{H^{1}}
=0
.
\end{equation}
Next, we define  
\begin{align}
\label{eq:4.7}
\Psi_{n}(x)
&:=
\Phi_{n,1}(M_{n,1}^{-\frac{2}{d-2}}x)-\Phi_{n,2}(M_{n,1}^{-\frac{2}{d-2}}x) 
			= M_{n,1} \Big\{  \TPhnone(x) - \mu_n \TPhntwo (\mu_{n}^{\frac{2}{d-2}} x)  \Big\}, 
\\[6pt]
\label{eq:4.8}
\wt{z}_{n}(x) &:=\frac{\Psi_{n}(x)}{\| \Psi_{n} \|_{L^{\infty}}}. 
\end{align}
Since $\widetilde{\Phi}_{n,1}$ and $\mu_n \TPhntwo (\mu_{n}^{\frac{2}{d-2}} x)$ are solutions to the same equation  
	\[
		-\Delta u + \alpha_{n,1} u = \beta_{n,1} u^p + u^{(d+2)/(d-2)},
	\]
we can verify that 
\begin{equation}\label{eq:4.9}
-\Delta \wt{z}_n = - \alpha_{n,1} \wt{z}_n + p \beta_{n,1} \int_0^1 V_n^{p-1} (x,\theta) d \theta \wt{z}_n 
+ \frac{d+2}{d-2} \int_0^1 V_n^{\frac{4}{d-2}} (x,\theta) d \theta \wt{z}_n 
\end{equation}
where 
\begin{equation}\label{eq:4.10}
V_{n}(x,\theta):=\theta \widetilde{\Phi}_{n,1}(x)+(1-\theta)\mu_{n} \widetilde{\Phi}_{n,2}(\mu_{n}^{\frac{2}{d-2}}x). 
\end{equation}
We first show that $\{\wt{z}_n\}$ is bounded in $\dHRd$:
\begin{lemma}\label{proposition:4.2}
	Assume $d\ge 5$ and $1 < p < \frac{d+2}{d-2}$. 
	Then, $\{\wt{z}_n\}$ is bounded in $\dHRd$. 
\end{lemma}

	\begin{proof}
Since $\wt{z}_n \in H^1(\Rd)$, using $\wt{z}_n$ as  test function to \eqref{eq:4.9}, we have 
	\begin{equation}\label{eq:4.11}
\begin{split}
\alpha_{n,1} \| \wt{z}_{n} \|_{L^{2}}^{2}
+
\| \nabla \wt{z}_{n} \|_{L^{2}}^{2} 
&\lesssim 
\beta_{n,1} \int_{\Rd} \int_{0}^1 V_n^{p-1}(x,\theta) d \theta \, |\wt{z}_{n}|^{2} \, dx 
\\[6pt]
& \quad + \int_{\Rd} \int_{0}^{1} V_{n}^{\frac{4}{d-2}} (x, \theta)
\, d \theta \, |\wt{z}_{n} |^{2} \, dx. 
\end{split}
\end{equation}
We see from Young's  inequality, \eqref{eq:4.4} and \eqref{eq:3.14} that for any $\delta >0$, 
\begin{equation}\label{eq:4.12}
\begin{split}
&\beta_{n,1}
\int_{\Rd} \int_{0}^1 V_n^{p-1}(x,\theta) d \theta \, |\wt{z}_{n}|^{2} 
\,dx 
\\[6pt]
\lesssim  \, &
\beta_{n,1}
\delta^{\frac{4(p-1)}{\gamma}} 
\int_{\Rd} |\wt{z}_{n}|^{2} 
\, dx
+
\delta^{-\frac{4}{d-2}} \beta_{n,1} 
\int_{\Rd} \int_{0}^{1} V_{n}^{\frac{4}{d-2}} (x,\theta) \, d \theta 
\, |\wt{z}_{n}|^{2} 
\, dx
\\[6pt]
\lesssim \, &
\alpha_{n,1}
\delta^{\frac{4(p-1)}{\gamma}}  \| \wt{z}_{n} \|_{L^{2}}^{2}
+
\delta^{-\frac{4}{d-2}} \beta_{n,1} 
\int_{\Rd} \int_{0}^{1} V_{n}^{\frac{4}{d-2}} (x,\theta) \, d \theta 
\, |\wt{z}_{n}|^{2} 
\, dx,
\end{split}
\end{equation}
where the implicit constants are independent of $\delta$.

	Next, set $\e_{0} :=\frac{4}{3(d-2)}$ and $p_{0} := \frac{2^\ast}{2-\e_0}$. Note that 
	\begin{equation*}
 1-\frac{1}{p_{0}} = \frac{4+(d-2)\e_0 }{2d} = \frac{8}{3d} .
	\end{equation*} 
Since $\|\wt{z}_n \|_{L^\infty} = 1$ holds by definition, it follows from H\"older's inequality, 
 Proposition \ref{proposition:3.1} and Sobolev's inequality that 
for all $n \geq 1$, 
	\begin{equation}\label{eq:4.13}
		\begin{aligned}
\int_{\Rd} \int_0^{1} V_{n}^{\frac{4}{d-2}} (x, \theta) d \theta |\wt{z}_n|^{2} dx 
&\le 
\int_{\Rd} \int_0^{1} V_{n}^{\frac{4}{d-2}} (x, \theta) d \theta |\wt{z}_n |^{2-\e_0} d x
			\\[6pt]
			 &\leq 
\left\| \int_0^{1} V_{n}^{\frac{4}{d-2}} (x, \theta) d \theta 
\right\|_{L^{\frac{2d}{4+(d-2)\varepsilon_{0}}}} 
\| |\wt{z}_{n}|^{2-\e_0} \|_{L^{p_{0}}}
			\\[6pt]
			&\lesssim  
\left\| (1+|x|)^{-4} \right\|_{L^{\frac{2d}{4+(d-2)\varepsilon_{0}}}}
\left\|\wt{z}_{n} \right\|_{L^{2^{*}}}^{2-\varepsilon_{0}}
\lesssim \| \nabla  \wt{z}_n \|_{L^2}^{2-\e_0}.
		\end{aligned}
	\end{equation}

Choosing $\delta$ sufficiently small depending on $d$ and $p$, 
we find from \eqref{eq:4.11}, \eqref{eq:4.12}, \eqref{eq:4.13} and $\lim_{n\to \infty}\beta_{n,1}=0$ that 
\begin{equation*}
\| \nabla \wt{z}_{n} \|_{L^{2}}^{2} 
\le  
C  \| \nabla  \wt{z}_n \|_{L^2}^{2-\e_0}
\end{equation*}
where $C>0$ depends only on $d$ and $p$. 
Hence, $\{ \wt{z}_{n} \}$ is bounded in $\dHRd$. 
	\end{proof}

Next, we derive a uniform decay estimate for $\{\wt{z}_n\}$. 

\begin{lemma}\label{proposition:4.3}
Assume $d\ge 5$ and $1<p<\frac{d+2}{d-2}$. Then, there exists $C_{1}>0$ such that 
for any $x \in \Rd$ and $n \geq 1$, 
\begin{equation}\label{eq:4.14}
|\wt{z}_{n}(x)| \leq C_{1} |x|^{-(d-2)}
.
\end{equation}
\end{lemma}
\begin{proof}
We see from \eqref{eq:3.16} and Lemma \ref{proposition:4.2} that $\{K[\wt{z}_{n}]  \}$ is bounded in $\dHRd$. 
Furthermore, it follows from \eqref{eq:4.9} that 
\begin{equation}\label{eq:4.15} 
\begin{split}
&-\Delta K[\wt{z}_{n}] 
+ \frac{\alpha_{n,1}}{|x|^{4}}  K[\wt{z}_{n}] 
	\\[6pt]	= \, & \frac{1}{|x|^{4}} 
		\left[ p \beta_{n,1}  \int_0^1 V_{n}^{p-1} \left(\frac{x}{|x|^{2}},\theta \right) d \theta + 
		\frac{d+2}{d-2} \int_0^1 V_{n}^{\frac{4}{d-2}}  \left(\frac{x}{|x|^{2}},\theta\right) d \theta  \right] K[\wt{z}_{n}].
\end{split} 
\end{equation}

To prove \eqref{eq:4.14}, we shall apply Proposition \ref{proposition:B.1} (ii). 
We first remark that by Lemma \ref{proposition:3.2}, 
$\Psi_n$ decays exponentially and so does $\wt{z}_n$. 
Thus, for any $n\ge 1$ and any $v \in H_{0}^{1}(B_{4})$, we have 
\begin{equation}\label{eq:4.16}
\int_{B_{4}} \frac{\alpha_{n,1}}{|x|^{4}}
|K[\wt{z}_{n}](x)||v(x)|\,dx <\infty
.
\end{equation}

	Next, it follows from \eqref{eq:4.10} that for each $r > 0$, 
$x \in \Rd \setminus \{0\}$ and $ \theta \in [0,1]$, 
	\begin{equation}\label{eq:4.17}
		V_n^r \left( \frac{x}{|x|^2} , \theta \right) \leq 
		C_r \left\{\widetilde{\Phi}_{n,1}^r \left( \frac{x}{|x|^2} \right)
		+ \mu_{n}^r \widetilde{\Phi}_{n,2}^r \left( \mu_{n}^{\frac{2}{d-2}} \frac{x}{|x|^2} \right)
		\right\}
	\end{equation}
where $C_r$ depends only on $r$. 
When $r = p - 1$, by $\gamma = 4 - (d-2)(p-1)$, we get 
	\[
		\begin{aligned}
			\frac{\beta_{n,1}}{|x|^4} 
			\int_0^1 V_n^{p-1} \left( \frac{x}{|x|^2} , \theta \right) d \theta 
			& \leq C_p \frac{\beta_{n,1}}{|x|^4} 
			\left\{ \wt{\Phi}_{n,1}^{p-1} \left( \frac{x}{|x|^2}  \right)
			+ \mu^{p-1} \wt{\Phi}_{n,2}^{p-1} 
			\left(  \mu_n^{\frac{2}{d-2}} \frac{x}{|x|^2} \right) \right\}
			\\
			& \leq C_p \frac{\beta_{n,1}}{|x|^{\gamma}}
			\left\{ K[ \wt{\Phi}_{n,1} ]^{p-1} (x) + 
			\mu_n^{-(p-1)} K [ \wt{\Phi}_{n,2} ]^{p-1} \left( \mu_n^{-\frac{2}{d-2}} x \right)
			  \right\}.
		\end{aligned}
	\]
Hence, recalling \eqref{eq:4.5}, \eqref{eq:3.28} in the proof of Proposition \ref{proposition:3.1}, we see that for some $q_0 > \frac{d}{2}$, 
	\begin{equation}\label{eq:4.18}
		\sup_{n \geq 1}\left\| \frac{\beta_{n,1}}{|x|^4} 
		\int_0^1 V_n^{p-1} \left( \frac{x}{|x|^2} , \theta \right) d \theta  \right\|_{L^{q_0}(B_4)} 
		< \infty.
	\end{equation}

	On the other hand, when $r = \frac{4}{d-2}$, \eqref{eq:4.17} gives us that 
	\[
		\frac{1}{|x|^4}V_n^{\frac{4}{d-2}} \left( \frac{x}{|x|^2} , \theta \right)
		 \leq 
		C_d \left\{ K[ \wt{\Phi}_{n,1} ]^{\frac{4}{d-2}} (x)
		+ K[ \wt{\Phi}_{n,2} ]^{\frac{4}{d-2}} \left( \mu_n^{\frac{2}{d-2}} x \right) \right\}.
	\]
By Proposition \ref{proposition:3.1}, $\{ K[\wt{\Phi}_{n,j}] \}$ is bounded in $L^\infty(\Rd)$ 
for $j =1,2$. Hence, 
	\begin{equation}\label{eq:4.19}
		\sup_{n \geq 1} 
		\left\| \frac{1}{|x|^4} \int_0^1V_n^{\frac{4}{d-2}} \left( \frac{x}{|x|^2} , \theta \right) 
		d \theta 
		 \right\|_{L^\infty(B_4)} < \infty. 
	\end{equation}

	From \eqref{eq:4.15}, \eqref{eq:4.16}, \eqref{eq:4.18} and \eqref{eq:4.19}, 
we apply Proposition \ref{proposition:B.1} (ii) to obtain 
	\[
		\sup_{n\geq 1} \left\| K[\tilde{z}_n] \right\|_{L^\infty(B_1)} < \infty.
	\]
Therefore, Lemma \ref{proposition:4.3} holds. 
\end{proof}

As a corollary of Lemma \ref{proposition:4.3}, 
we obtain

\begin{lemma}\label{proposition:4.4}
Assume $d\ge 5$ and $1<p<\frac{d+2}{d-2}$. Then 
\begin{equation*}
\sup_{n\ge 1}\|\wt{z}_{n}\|_{L^{2}}<\infty. 
\end{equation*}
\end{lemma}

Before proving Theorem \ref{theorem:1.1}, 
we use the following identity which is easily obtained from elementary calculations: 
\begin{lemma}\label{proposition:4.5}
Assume $d\ge 5$ and $1\le q \le \frac{d+2}{d-2}$. Then, the following holds:
	\begin{align*}
	\int_{\mathbb{R}^{d}} W^{q}(x) \Lambda W(x) \,dx 
	&=
	-\frac{4-(d-2)(q-1)}{2(q+1)}\|W\|_{L^{q+1}}^{q+1}
	\end{align*}
where $\Lambda$ is defined in \eqref{11111}.
\end{lemma}

	Now, we derive a contradiction and prove the uniqueness part:

\begin{proof}[Proof of uniqueness in Theorem \ref{theorem:1.1}]
By Lemmas \ref{proposition:4.2} and \ref{proposition:4.4}, 
$\{\wt{z}_n\}$ is bounded in $H^1(\Rd)$ and we may assume that 
\begin{equation}
\label{eq:4.20}
	\lim_{n\to \infty}\wt{z}_n = \wt{z}_\infty \quad 
	{\rm weakly \ in\ } H^1(\Rd). 
\end{equation}
Moreover, recalling that  $\wt{z}_n$ satisfies \eqref{eq:4.9}, 
by elliptic regularity with Corollary \ref{theorem:3.1},  \eqref{eq:4.3} and \eqref{eq:4.6}, we can see that 
	\begin{equation}\label{eq:4.21}
	\lim_{n\to \infty} \wt{z}_n = \wt{z}_\infty \quad {\rm in} \ C^2_{\rm loc}(\Rd),
	\end{equation}
hence, $\wt{z}_\infty$ is a solution to
\begin{equation}\label{eq:4.22}
-\Delta \wt{z}_{\infty}
-
\frac{d+2}{d-2} W^{\frac{4}{d-2}}
\wt{z}_{\infty}
=
0
.
\end{equation}
Thus, we find from \eqref{eq:4.22}, \eqref{eq:1.11} and the radial symmetry of $\wt{z}_{\infty}$ that 
either $\wt{z}_{\infty}\equiv 0$ or $\wt{z}_{\infty} = \kappa \Lambda W$ with $\kappa \neq 0$.

	First, we suppose that $\wt{z}_{\infty}\equiv 0$. Then, it follows from \eqref{eq:4.8} and \eqref{eq:4.21} that 
$1= \| \wt{z}_n \|_{L^\infty} \to 0$, which is a contradiction.

	Next, assume $\wt{z}_{\infty}=\kappa \Lambda W$ for some $\kappa \neq 0$. 
 Using \eqref{eq:2.4} and \eqref{eq:2.6}, we see that  
\begin{equation}\label{eq:4.23}
\begin{split}
&\frac{M_{n,1}}{\|\Psi_{n}\|_{L^{\infty}}}
\frac{\omega_{n}}{d}
\bigm(
\|\Phi_{n,1}\|_{L^{2}}^{2}-\|\Phi_{n,2}\|_{L^{2}}^{2}
\bigm)
\\[6pt]
= \, &
\frac{M_{n,1}}{\|\Psi_{n}\|_{L^{\infty}}}
\frac{4-(d-2)(p-1)}{2d(p+1)}
\bigm(
\|\Phi_{n,1}\|_{L^{p+1}}^{p+1}-\|\Phi_{n,2}\|_{L^{p+1}}^{p+1}
\bigm)
.
\end{split}
\end{equation}
For the left-hand side of \eqref{eq:4.23}, 
using the change of variables, we observe that 
\begin{equation}\label{eq:4.24}
\begin{aligned}
&\frac{M_{n,1}}{\|\Psi_{n}\|_{L^{\infty}}}
\frac{\omega_{n}}{d}
\bigm(
\|\Phi_{n,1}\|_{L^{2}}^{2}-\|\Phi_{n,2}\|_{L^{2}}^{2}
\bigm)
\\[6pt]
=&\,
M_{n,1}\frac{\omega_{n}}{d}
\int_{\mathbb{R}^{d}}
\frac{\Phi_{n,1}(x)-\Phi_{n,2}(x)}{\|\Psi_{n}\|_{L^{\infty}}}  
(\Phi_{n,1}(x)+\Phi_{n,2}(x))\,dx  
\\[6pt]
=&\,
\frac{\alpha_{n,1}}{d}
\int_{\mathbb{R}^{d}}
\widetilde{z}_{n}(x)
\Big[\wt{\Phi}_{n,1}(x)
+
\mu_{n} \wt{\Phi}_{n,2}(\mu_{n}^{\frac{2}{d-2}} x )\Big] \,dx 
.
\end{aligned}
\end{equation}
In a similar way, the right-hand side of \eqref{eq:4.23} becomes
\begin{equation}\label{eq:4.25}
\begin{aligned}
&
\frac{M_{n,1}}{\|\Psi_{n}\|_{L^{\infty}}}
\frac{4-(d-2)(p-1)}{2d(p+1)}
\bigm(
\|\Phi_{n,1}\|_{L^{p+1}}^{p+1}-\|\Phi_{n,2}\|_{L^{p+1}}^{p+1}
\bigm)
\\[6pt]
=&\,
M_{n,1}
\frac{4-(d-2)(p-1)}{2d}
\int_{\mathbb{R}^{d}} \int_{0}^{1}\Big[ \theta \Phi_{n,1}+(1-\theta)\Phi_{n,2}\Big]^{p}d\theta
\frac{\Phi_{n,1}-\Phi_{n,2}}{\|\Psi_{n}\|_{L^{\infty}}}
\,dx 
\\[6pt]
=&\,
\beta_{n,1} \frac{4-(d-2)(p-1)}{2d}
\int_{\mathbb{R}^{d}} \int_{0}^{1}\Big[ \theta \wt{\Phi}_{n,1}(x)
+
(1-\theta) \mu_{n} \wt{\Phi}_{n,2}(\mu_{n}^{\frac{2}{d-2}} x )  \Big]^{p}d\theta \widetilde{z}_{n}(x)\,dx 
.
\end{aligned}
\end{equation}
Hence, we obtain the following identity from \eqref{eq:4.23} through \eqref{eq:4.25} 
with \eqref{eq:4.10}:
\begin{equation}\label{eq:4.26}
	\begin{aligned}
		&\frac{1}{2}
		\int_{\mathbb{R}^{d}}
		\wt{z}_{n}(x)
		\Big[\wt{\Phi}_{n,1}(x)
		+
		\mu_{n} \wt{\Phi}_{n,2}(\mu_{n}^{\frac{2}{d-2}} x )\Big] \,dx 
		\\
		= & 
		\frac{\beta_{n,1}}{\alpha_{n,1}}
		\frac{4-(d-2)(p-1)}{4}
		\int_{\mathbb{R}^{d}} \int_{0}^{1} V_{n}(x,\theta)^{p}d\theta \wt{z}_{n}(x)\,dx .
	\end{aligned}
\end{equation}

	Since Corollary \ref{theorem:3.1} and Lemma \ref{proposition:4.1} yield 
	\[
		\lim_{n\to \infty} \TPhnone = W, \quad \lim_{n\to \infty} \mu_{n} \wt{\Phi}_{n,2}(\mu_{n}^{\frac{2}{d-2}} \cdot  ) 
		= W \quad {\rm strongly \ in} \ L^2(\Rd),
	\]
recalling \eqref{eq:4.20} with $\wt{z}_\infty = \kappa \Lambda W$, 
 we see from Lemma \ref{proposition:4.5} that 
\begin{equation}\label{eq:4.27}
\begin{split}
&
\lim_{n\to \infty}
\frac{1}{2}
\int_{\mathbb{R}^{d}}
\widetilde{z}_{n}(x)
\left[\wt{\Phi}_{n,1}(x)+\mu_{n} \wt{\Phi}_{n,2}( \mu_{n}^{\frac{2}{d-2}} x )\right] \,dx 
=
\int_{\mathbb{R}^{d}}\kappa \Lambda W(x) W(x) \,dx 
=-\kappa \|W\|_{L^{2}}^{2}.
\end{split}
\end{equation}
In a similar way, we can check that 

\[
	\lim_{n\to \infty} 	\int_{0}^{1} V_n^p (x,\theta) \,d\theta = W^{p} 
\quad  \mbox{strongly in $L^2(\Rd)$}
	\]
and 
\begin{equation}\label{eq:4.28}
\begin{split}
\lim_{n\to \infty}
\int_{\mathbb{R}^{d}} \int_{0}^{1}
V_n(x,\theta)^p 
d\theta \wt{z}_{n}(x)\,dx
&=
\int_{\mathbb{R}^{d}} 
W^{p}(x) \kappa \Lambda  W(x)\,dx 
\\[6pt]
&=
-\kappa 
\frac{4-(d-2)(p-1)}{2(p+1)}\|W\|_{L^{p+1}}^{p+1}
.
\end{split}
\end{equation}
Putting \eqref{eq:4.26}, \eqref{eq:4.27}, \eqref{eq:4.2}  and \eqref{eq:4.28} together, we find that 
\begin{equation*}
\begin{split}
-\kappa \|W\|_{L^{2}}^{2}
&=
\frac{2(p+1)}{4-(d-2)(p-1)}
\frac{\| W\|_{L^{2}}^2}{\|W\|_{L^{p+1}}^{p+1}}
\\[6pt]
&\qquad \times 
\frac{4-(d-2)(p-1)}{4} 
\left\{ -\kappa 
\frac{4-(d-2)(p-1)}{2(p+1)}\|W\|_{L^{p+1}}^{p+1} \right\}
\\[6pt]
&= -\kappa \frac{4 - (d-2)(p-1)}{4}  \|W\|_{L^{2}}^{2}
.
\end{split}
\end{equation*}
This contradicts $\kappa \neq 0$, and 
the uniqueness in Theorem \ref{theorem:1.1} holds. 
\end{proof}

	Next, we shall prove the nondegeneracy in $\Hrad(\Rd)$.

	\begin{proof}[Proof of nondegeneracy in Theorem \ref{theorem:1.1}]
From the uniqueness part, 
there exists $\wt{\omega}_\ast>0$ such that 
if $\omega > \wt{\omega}_\ast$, then \eqref{eq:1.1} admits a unique radial 
positive ground state and we denote it by $\Pho$. Our aim is to 
find $\omega_\ast \geq \wt{\omega}_\ast$ such that 
\begin{equation}\label{eq:4.29}
	{\rm Ker} \, L_{\Phi_\omega} |_{\Hrad(\Rd)} = \left\{ 0 \right\} 
	\quad {\rm for\ every} \ \omega > \omega_\ast. 
\end{equation}

In order to prove \eqref{eq:4.29}, 
we argue indirectly and suppose to the contrary that there exist 
$\{ \omega_n \}$ and $\{u_n\}$ such that 
	\[
		\wt{\omega}_\ast < \omega_n \to \infty, \quad 
		u_n \in H^1_{\rm rad} (\Rd) \setminus \left\{ 0 \right\}, 
		\quad L_{\Phon} u_n = 0.
	\]
Remark that we may assume $u_n: \Rd \to \R$ and $\| u_n \|_{L^\infty(\Rd)} = 1$ 
due to the linearity of $L_{\Phon}$ and elliptic regularity. 
Set 
	\[
		v_n(x) := u_n \left( M_n^{-\frac{2}{d-2}} x \right), \quad 
		M_n := \Phon (0) = \| \Phon \|_{L^\infty}. 
	\]
Then we observe that 
	\begin{equation}\label{eq:4.30}
		-\Delta v_n + \alpha_n v_n = 
		\left[ 
		p \beta_n \TPhon^{p-1} + \frac{d+2}{d-2} 
		\TPhon^{\frac{4}{d-2}} 		
		\right] v_n \quad {\rm in} \ \Rd, \quad 
		\| v_n \|_{L^\infty} = 1
	\end{equation}
where $\alpha_n := \omega_n M_{n}^{-\frac{4}{d-2}}$, 
$\beta_n := M_n^{p-1-\frac{4}{d-2}}$ and 
$\TPhon (x) := \Phon ( M_n^{-\frac{2}{d-2}} x )$. 
By Proposition \ref{proposition:3.1}, 
we can argue as in Lemma 4.2 to show that 
	\begin{equation}\label{eq:4.31}
		\text{$\{v_n\}$ is bounded in $\dHRd$}. 
	\end{equation}

Next, let us consider the Kelvin transform of $v_n$ and 
write $K[v_n]$ for it. Then $K[v_n]$ satisfies 
	\[
		-\Delta K[v_n] + \alpha_n \frac{1}{|x|^4} K[v_n] 
		= \left\{ p \frac{\beta_n}{|x|^{\gamma}} K[\TPhon]^{p-1} 
		+ \frac{d+2}{d-2} K[ \TPhon ]^{\frac{4}{d-2}} \right\} K [v_n],
	\]
where $\gamma=4-(d-2)(p-1)$. We remark that each $v_n$ has an exponential decay and 
this fact can be proved reasoning as for Lemma \ref{proposition:3.2}. 
Therefore, applying the argument in Lemma 4.3, 
we get the uniform decay estimate for $\{v_n\}$: 
	\begin{equation}\label{eq:4.32}
		\sup_{n \geq 1} \left| v_n(x) \right| \leq C_0 \left( 1 + |x| \right)^{-(d-2)}
		\quad {\rm for\ all} \ x \in \Rd. 
	\end{equation}

	By \eqref{eq:4.31} and \eqref{eq:4.32}, $\{v_n\}$ is bounded in $H^1(\Rd)$ and 
we may assume that there exists a $v_\infty \in \Hrad (\Rd)$ so that 
	\[
		v_n \rightharpoonup v_\infty \quad \text{weakly in $H^1(\Rd)$}. 
	\]
Using Proposition \ref{proposition:2.1}, \eqref{eq:4.30}, \eqref{eq:4.32} and elliptic regularity, we have 
	\[
	 \lim_{n\to \infty}v_n = v_\infty \quad \mbox{strongly in $C^2_{\rm loc}(\Rd)$}, \quad 1 = \| v_\infty \|_{L^\infty}, \quad
		-\Delta v_\infty = \frac{d+2}{d-2} W^{\frac{4}{d-2}} v_\infty \quad {\rm in} \ \Rd. 
	\]
Since $\| v_\infty \|_{L^\infty} = 1$ and $v_\infty \in \Hrad (\Rd)$, 
from \eqref{eq:1.11}, there exists a $\kappa \neq 0$ such that 
	\begin{equation}\label{eq:4.33}
		v_\infty = \kappa \Lambda W. 
	\end{equation}

Next, we consider $w_n(x) := x \cdot \nabla \TPhon (x)$. 
It is not difficult to check that $w_n$ satisfies 
\begin{equation}\label{eq:4.34}
-\Delta w_n + \alpha_n w_n = 
\left[  p \beta_n \TPhon^{p-1} 
+ \frac{d+2}{d-2} \TPhon^{\frac{4}{d-2}} \right] w_n 
+ 2 \left[ - \alpha_n \TPhon + \beta_n \TPhon^p 
+ \TPhon^{\frac{d+2}{d-2}} \right].
\end{equation}
Thus, multiplying \eqref{eq:4.30} by $w_n$ and \eqref{eq:4.34} by $v_n$, 
it follows from the integration by parts that 
\begin{equation}\label{eq:4.35}
\int_{\Rd}\left[ - \alpha_n \TPhon + \beta_n \TPhon^p 
+ \TPhon^{\frac{d+2}{d-2}} \right] v_n dx = 0.
\end{equation}
Recall that  
\begin{equation}\label{17/10/04/15:56}
-\Delta \TPhon + \alpha_n \TPhon = 
\beta_n \TPhon^p + \TPhon^{\frac{d+2}{d-2}} \quad {\rm in} \ \Rd. 
\end{equation}
Multiply \eqref{17/10/04/15:56} by $v_{n}$ and \eqref{eq:4.30} by $\TPhon$, and  then integrate them: putting these together, we obtain  
	\[
		\begin{aligned}
			\int_{\Rd} 
			\left[ \beta_n \TPhon^p + \TPhon^{\frac{d+2}{d-2}} \right] v_n dx 
			&= \int_{\Rd} \nabla \TPhon \cdot \nabla v_n + \alpha_n \TPhon v_n dx 
			\\
			&= \int_{\Rd} 
			\left[ 
			p \beta_n \TPhon^{p-1} + \frac{d+2}{d-2} 
			\TPhon^{\frac{4}{d-2}} 		
			\right] v_n \TPhon dx,
		\end{aligned}
\]
which implies 
\[
\int_{\Rd} \TPhon^{\frac{d+2}{d-2}} v_n d x 
= - \frac{(d-2)(p-1)}{4} \int_{\Rd} \beta_n \TPhon^p v_n dx. 
\]
Combining this with \eqref{eq:4.35}, we find
\[
\int_{\Rd} \TPhon v_n dx 
= \left[ 1 - \frac{(d-2)(p-1)}{4} \right] \frac{\beta_n}{\alpha_n} \int_{\Rd} 
\TPhon^p v_n dx.
\]
As $n \to \infty$, Corollary \ref{theorem:3.1}, Lemma \ref{proposition:4.1} and \eqref{eq:4.33} yield   
\[
\int_{\Rd} W \kappa \Lambda W d x 
= \left[ 1 - \frac{(d-2)(p-1)}{4} \right] 
\frac{2(p+1)}{4-(d-2)(p-1)} \frac{\|W\|_{L^2}^2}{\| W \|_{L^{p+1}}^{p+1}} 
\int_{\Rd} W^p \kappa \Lambda W d x.
\]
Since $\kappa \neq 0$, Lemma 4.5 gives a contradiction:
\[
- \| W \|_{L^2}^2 = - \left[ 1 - \frac{(d-2)(p-1)}{4} \right] \| W \|_{L^2}^2. 
\]
Thus, \eqref{eq:4.29} holds and we complete the proof of Theorem \ref{theorem:1.1}. 
	\end{proof}

\appendix

%%%%%%%%%%%%%%%%%%%%%%%%%%%%%%%%%%%%%%%%%%%%%%%%%%%%%%%%%%%%%%%%%%%%%%%%%%%%%%%

\section{Existence of ground state}\label{section:A}
In this section, we sketch the proof of Proposition \ref{proposition:1.1}. Since we restrict nonlinearities to combined power-type ones, the proof is much simpler than the general case dealt with in \cite{Zhang-Zou}. In particular, we can use a positive functional $\mathcal{I}_{\omega}$ given by  
\begin{equation}\label{eq:A.1}
\begin{split}
\mathcal{I}_{\omega}(u)
&:=
\mathcal{S}_{\omega}(u)-\frac{1}{p+1}\mathcal{N}_{\omega}(u)
\\[6pt]
&=
\frac{p-1}{2(p+1)} 
\Big\{  \|\nabla u \|_{L^{2}}^{2} + \omega \|u \|_{L^{2}}^{2}   \Big\} 
+
\frac{4-(d-2)(p-1)}{2d(p+1)} 
\|u \|_{L^{2^{*}}}^{2^{*}}.
\end{split}
\end{equation}
Moreover, we easily verify the following structures of $\mathcal{S}_{\omega}$ and $\mathcal{N}_{\omega}$ 
(cf. \cite[Chapter 4]{Willem}):

\noindent 
\textbullet~For any $u \in H^{1}(\mathbb{R}^{d})\setminus \{0\}$, 
there exists a unique $\lambda(u)>0$ such that 
\begin{equation}
\label{eq:A.2}
\mathcal{N}_{\omega}(\lambda u)
\left\{
\begin{array}{rcl}
>0 & \mbox{if}& 0< \lambda <\lambda(u),
\\
=0 & \mbox{if}& \lambda=\lambda(u),
\\
<0 & \mbox{if}& \lambda >\lambda(u).
\end{array} 
\right.
\end{equation}

\noindent
\textbullet~For any $u \in H^{1}(\mathbb{R}^{d})\setminus \{0\}$, 
\begin{equation}
\label{eq:A.3}
\text{the function $\lambda \mapsto \mathcal{I}_\omega (\lambda u)$ 
is non-decreasing in $[0,\infty)$}. 
\end{equation}

Next, we introduce several variational values: 
\begin{align}
\label{eq:A.4}
\sigma
&:=
\inf\{ \|\nabla u \|_{L^{2}}^{2} : \mbox{$u \in \dot{H}^{1}(\mathbb{R}^{d})$ 
with $\|u\|_{L^{2^{*}}}=1$}\},
\\[6pt]
\label{17/07/08/11:28}
m_{\omega}
&:=\inf\{ \mathcal{S}_{\omega}(u) : \mbox{$u\in H^{1}(\mathbb{R}^{d})\setminus \{0\}$ with $\mathcal{N}_{\omega}(u)=0$} \},
\\[6pt]
\nonumber
\widetilde{m}_{\omega}
&:=\inf \big\{\mathcal{I}_{\omega}(u) : 
\mbox{$u\in H^{1}(\mathbb{R}^{d})\setminus \{0\}$ with $\mathcal{N}_{\omega}(u)\le 0$} 
\big\}
.
\end{align}

	By a standard argument (cf. \cite[Chapter 4]{Willem}), 
it is known that a minimizer for $m_\omega$ becomes a ground state to \eqref{eq:1.1}. 
Hence, in order to prove Proposition \ref{proposition:1.1}, 
it suffices to show the existence of minimizer for $m_\omega$.

	We first state the relationship between $m_\omega$ and $\wt{m}_\omega$ 
	(cf. \cite[Proposition 1.2]{Akahori-Ibrahim-Kikuchi-Nawa1}):

\begin{lemma}\label{proposition:A.1}
Assume $d\ge 3$ and $1<p<\frac{d+2}{d-2}$. Then, for any $\omega>0$, we have the following:
\begin{enumerate}
\item[\emph{(i)}] 
$m_{\omega}=\widetilde{m}_{\omega}>0$
\item[\emph{(ii)}] 
Any minimizer for $\widetilde{m}_{\omega}$ is also a minimizer 
for $m_{\omega}$, and vice versa.
\end{enumerate}
\end{lemma}
\begin{proof}
We shall prove claim (i). 
Since $\mathcal{I}_\omega (u) = \mathcal{S}_\omega(u)$ for every $u \in H^1(\Rd)$ 
with $\mathcal{N}_\omega (u) = 0$, it is clear that $\wt{m}_\omega \leq m_\omega$. For the opposite inequality $m_\omega \leq \wt{m}_\omega$, 
fix any $u \in H^1(\Rd) \setminus \{0\}$ with $\mathcal{N}_\omega (u) \leq 0$. 
By \eqref{eq:A.2}, there exists a $\lambda \in (0,1]$ such that 
$\mathcal{N}_\omega (\lambda u) = 0$. By \eqref{eq:A.3}, 
	\[
		m_\omega \leq \mathcal{S}_\omega (\lambda u) = \mathcal{I}_\omega(\lambda u) 
		\leq \mathcal{I}_\omega (u),
	\]
which yields $m_\omega \leq \wt{m}_\omega$. Thus, $m_\omega = \wt{m}_\omega$. 
 It remains to prove that $\wt{m}_\omega >0$. Let $u \in H^{1}(\mathbb{R}^{d})\setminus \{0\}$ with $\mathcal{N}_{\omega}(u)\le 0$. Then, it follows from $\mathcal{N}_{\omega}(u)\le 0$ and Sobolev's inequality that 
\[
\min\{1, \omega \}
\|u\|_{H^{1}}^{2}
\lesssim  
\|u\|_{H^{1}}^{p+1}+\|u\|_{H^{1}}^{2^{*}}
.
\]
This implies that there exists a constant $c(\omega)>0$ such that $\|u\|_{H^{1}}^{2}\ge c(\omega)$ and therefore $\mathcal{I}_\omega (u) \gtrsim \min\{1,\omega\}c(\omega)$. Since $u$ is arbitrary, we find that $\wt{m}_\omega >0$.

Next, we shall prove claim {\rm (ii)}. Since $\mathcal{I}_{\omega}=\mathcal{S}_{\omega}-\frac{1}{p+1}\mathcal{N}_{\omega}$ and $m_{\omega}=\widetilde{m}_{\omega}$, it suffices to prove that $\mathcal{N}_{\omega}(\widetilde{Q}_{\omega})=0$ for all minimizer $\widetilde{Q}_{\omega}$ for $\widetilde{m}_{\omega}$. Suppose the contrary that there exists a minimizer $\widetilde{Q}_{\omega}$ for $\widetilde{m}_{\omega}$ such taht $\mathcal{N}_{\omega}(\widetilde{Q}_{\omega})<0$. Then, it follows from \eqref{eq:A.2} that there exists a unique $\lambda_{0}\in (0,1)$ such that $\mathcal{N}_{\omega}(\lambda_{0}\widetilde{Q}_{\omega})=0$. Furthermore, we have 
\[
\widetilde{m}_{\omega}\le \mathcal{I}_{\omega}(\lambda_{0}\widetilde{Q}_{\omega})
< \mathcal{I}_{\omega}(\widetilde{Q}_{\omega})=\widetilde{m}_{\omega},
\]
which is a contradiction. Thus, $\mathcal{N}_{\omega}(\widetilde{Q}_{\omega})=0$. 
\end{proof}

	Next, we state a key inequality to show the existence of minimizer for $m_\omega$ (cf. \cite[Lemma 2.2]{Zhang-Zou}):

\begin{lemma}\label{proposition:A.2}
Assume that $d\ge 3$ and $3< p<5$, or $d\ge 4$ and $1<p<\frac{d+2}{d-2}$. Then, the following estimate holds 
\begin{equation}\label{eq:A.5}
m_{\omega}< \frac{1}{d}\sigma^{\frac{d}{2}} 
= \frac{1}{2} \| \nabla W \|_{L^2}^2 - \frac{1}{2^\ast} \| W \|_{L^{2^\ast}}^{2^\ast} 
= \frac{1}{d} \| \nabla W \|_{L^2}^2 
.
\end{equation}
\end{lemma}
\begin{proof}
Let $\chi$ be an even smooth function on $\mathbb{R}$ such that $\chi(r)=1$ for $0\le r\le 1$, $\chi(r)=0$ for $r\ge 2$, 
and $\chi$ is non-increasing on $[0,\infty)$. Then, we define 
\[
W_\e(x) := \varepsilon^{-\frac{d-2}{4}}W \left( \frac{x}{\sqrt{\varepsilon}} \right)
	=
	\varepsilon^{\frac{d-2}{4}}\left(\varepsilon+\frac{|x|^{2}}{d(d-2)}\right)^{-\frac{d-2}{2}}, 
\quad V_{\varepsilon}(x)
:=
\chi(|x|)W_{\varepsilon}(x).
\]
Then, we can verify that 
$\sigma^{\frac{d}{2}} = \| \nabla W_\e \|_{L^2}^2 = \| W_\e \|_{L^{2^\ast}}^{2^\ast} $ and 
\begin{align}
\label{17/07/08/13:27}
\|\nabla V_{\varepsilon}\|_{L^{2}}^{2}
&=
\sigma^{\frac{d}{2}}+O(\varepsilon^{\frac{d-2}{2}})
,
\\[6pt]
\label{17/07/08/14:29}
\| V_{\varepsilon} \|_{L^{2^{*}}}^{2^{*}}
&=
\sigma^{\frac{d}{2}}+O(\varepsilon^{\frac{d}{2}})
.
\end{align}
Moreover, we find that 
\begin{equation}\label{17/07/08/20:47}
\|V_{\varepsilon}\|_{L^{q+1}}^{q+1} 
=
\left\{ \begin{array}{rcl}
O (\varepsilon^{\frac{2d-(d-2)(q+1)}{4}})
&\mbox{if}&  
\frac{2}{d-2} < q, 
\\[6pt]
O(\varepsilon^{\frac{d}{4}}|\log{\varepsilon}|)
&\mbox{if}& 
q=\frac{2}{d-2},
\\[6pt]
O(\varepsilon^{\frac{(d-2)(q+1)}{4}})
&\mbox{if}&
0< q< \frac{2}{d-2}
.
\end{array} \right.
\end{equation}

Next, for a given $\varepsilon\in (0,1)$, we introduce a function $y_{\varepsilon} \colon (0,\infty) \to \mathbb{R}$ as 
\[
y_{\varepsilon}(t):=\frac{1}{2}t^{2} \big\{ 
\|\nabla V_{\varepsilon}\|_{L^{2}}^{2}
+
\omega \|V_{\varepsilon}\|_{L^{2}}^{2}
\big\} 
-\frac{t^{2^{*}}}{2^{*}} 
\|V_{\varepsilon}\|_{L^{2^{*}}}^{2^{*}}
.
\]
It is easy to verify that the function $y_{\varepsilon}$ attains its maximum only at the point
\[
\tau_{\varepsilon, \max}:=\frac{\big\{ 
\|\nabla V_{\varepsilon}\|_{L^{2}}^{2}
+
\omega \|V_{\varepsilon}\|_{L^{2}}^{2}
\big\}^{\frac{d-2}{4}} }{ \|V_{\varepsilon}\|_{L^{2^{*}}}^{\frac{d}{2}}
}.
\]
It follows from the definition of $\sigma$ (see \eqref{eq:A.4}), \eqref{17/07/08/13:27} and \eqref{17/07/08/14:29} that
\[
\frac{\|\nabla V_{\varepsilon}\|_{L^{2}}^{2}}{\|V_{\varepsilon}\|_{L^{2^{*}}}^{2}}
=
\frac{\sigma^{\frac{d}{2}} + O( \varepsilon^{ \frac{d-2}{2}} )}{\sigma^{\frac{d-2}{2}}+O(\varepsilon^{\frac{d}{2}})}
=
\sigma +  O(\varepsilon^{\frac{d-2}{2}}).
\]
Moreover, we see from \eqref{17/07/08/13:27} and \eqref{17/07/08/20:47} that 
\[
\frac{\|V_{\varepsilon}\|_{L^{2}}^{2}}{\|\nabla V_{\varepsilon}\|_{L^{2}}^{2}}
=
\left\{ \begin{array}{rcl} 
O(\varepsilon^{\frac{1}{2}})
&\mbox{if}& d=3,
\\[6pt]
O( \varepsilon|\log{\varepsilon}| )
&\mbox{if}& d=4,
\\[6pt]
O( \varepsilon)
&\mbox{if}& d\ge 5.
\end{array}\right.
\]
Hence, we find that 
\begin{equation}\label{17/07/08/16:06}
\begin{split}
y(\tau_{\varepsilon,\max})
&=
\frac{1}{d} 
\frac{\|\nabla V_{\varepsilon}\|_{L^{2}}^{d}
}{\|V_{\varepsilon}\|_{L^{2^{*}}}^{d}}
\bigg( 1+ \omega \frac{\|V_{\varepsilon}\|_{L^{2}}^{2}}{ \|\nabla V_{\varepsilon}\|_{L^{2}}^{2}}
\bigg)^{\frac{d}{2}}
\\[6pt]
&=
\left\{ \begin{array}{rcl}
\frac{1}{d}\sigma^{\frac{d}{2}} + O(\varepsilon^{\frac{1}{2}})
&\mbox{if}& d=3,
\\[6pt]
\frac{1}{d}\sigma^{\frac{d}{2}} + O(\varepsilon |\log{\varepsilon}|)
&\mbox{if}& d=4,
\\[6pt]
\frac{1}{d}\sigma^{\frac{d}{2}} + O(\varepsilon)
&\mbox{if}& d\ge 5,
\end{array} \right. 
\end{split}
\end{equation}

On the other hand, for each $\varepsilon\in (0,1)$, there exists $\tau_{\varepsilon,0}>0$ such that $\mathcal{N}_{\omega}(\tau_{\varepsilon,0}V_{\varepsilon})=0$. 
\par 
Now, we assume that $d=3$ and $3< p< 5$. Then, it follows from \eqref{17/07/08/13:27}, \eqref{17/07/08/14:29} and \eqref{17/07/08/20:47} that 
\[
\begin{split}
0
&=\mathcal{N}_{\omega}(\tau_{\varepsilon,0}V_{\varepsilon})
\\[6pt]
&=
\omega 
\tau_{\varepsilon,0}^{2}O(\varepsilon^{\frac{1}{2}})
+
\tau_{\varepsilon,0}^{2}
\{\sigma^{\frac{3}{2}}+O(\varepsilon^{\frac{1}{2}})\}
-
\tau_{\varepsilon,0}^{p+1} 
O(\varepsilon^{\frac{5-p}{4}})
-
\tau_{\varepsilon,0}^{6} 
\{\sigma^{\frac{3}{2}}+O(\varepsilon^{\frac{3}{2}})\}.
\end{split}
\]
Divide both sides above by $\tau_{\varepsilon,0}^{2}\{\sigma^{\frac{3}{2}}+O(\varepsilon^{\frac{3}{2}})\}$. Then, we obtain  
\[
\tau_{\varepsilon,0}^{4}
=
\omega O(\varepsilon^{\frac{1}{2}})
+
1+O(\varepsilon^{\frac{1}{2}})
-
\tau_{\varepsilon,0}^{p-1}O(\varepsilon^{\frac{5-p}{4}})
.
\]
Since $p-1 < 4$, this implies that for any $\omega>0$, 
\begin{equation}\label{17/07/08/21:59}
\lim_{\varepsilon \to 0}\tau_{\varepsilon,0} \ge \frac{1}{2}.
\end{equation}
Furthermore, it follows from the definition of $m_{\omega}$ (see \eqref{17/07/08/11:28}), \eqref{17/07/08/20:47}, \eqref{17/07/08/16:06}, \eqref{17/07/08/21:59} and $2< p <2^{*}-1$ that 
\[
\begin{split}
m_{\omega} 
&\le \mathcal{S}_{\omega}(\tau_{\varepsilon,0} V_{\varepsilon})
=
y_{\varepsilon}(\tau_{\varepsilon,0})
-
\frac{\tau_{\varepsilon,0}^{p+1}}{p+1}
\|V_{\varepsilon}\|_{L^{p+1}}^{p+1}
\\[6pt]
&\le 
y_{\varepsilon}(\tau_{\varepsilon,\max})
-
\frac{\tau_{\varepsilon,0}^{p+1}}{p+1}
c_{1} \varepsilon^{\frac{5-p}{4}}
=
\frac{1}{3}\sigma^{\frac{3}{2}} + O(\varepsilon^{\frac{1}{2}})
-c_{2} \varepsilon^{\frac{5-p}{4}}
\end{split} 
\]
for some positive constants $c_{1}$ and $c_{2}$ depending only on $p$. Thus, we find that if $p>3$ and $\varepsilon$ is sufficiently small depending only on $p$ and $\omega$, then 
\[
m_{\omega} < \frac{1}{3}\sigma^{\frac{3}{2}}.
\]
Similarly, we can prove that if $d\ge 4$, then claim \eqref{eq:A.5} is true. 
\end{proof}

Now, we are ready to prove Proposition \ref{proposition:1.1}.
\begin{proof}[Proof of Proposition \ref{proposition:1.1}]
By Lemma \ref{proposition:A.1}, it suffices to prove the existence of minimizer for $\widetilde{m}_{\omega}$. 
To this end, we consider  a minimizing sequence $\{u_{n}\}$ for $\widetilde{m}_{\omega}$. 
We denote the Schwarz symmetrization of $u_{n}$ by $u_{n}^{*}$. 
Note that $\| \nabla u_n^\ast \|_{L^2} \leq \| \nabla u_n \|_{L^2}$ and 
$\| u_n^\ast \|_{L^q} = \| u_n \|_{L^q}$ hold for each $q \in [2,2^\ast]$. 
For example, see \cite{Lieb-Loss}. From these properties, we have
\begin{align}
\label{eq:A.6}
&\mathcal{N}_{\omega}(u_{n}^{*})\le 0
\quad 
\mbox{for any $n \ge 1$},
\\[6pt]
\label{eq:A.7}
&\lim_{n\to \infty}\mathcal{I}_{\omega}(u_{n}^{*})=\widetilde{m}_{\omega},
\\[6pt]
\nonumber
& 
\left\| u_{n}^{*} \right\|_{H^{1}} < \infty
\quad 
\mbox{for any $n \ge 1$}.
\end{align}
Since $\{u_{n}^{*}\}$ is  radially symmetric and bounded in $H^{1}(\mathbb{R}^{d})$, there exists a radially symmetric function $Q \in H^{1}(\mathbb{R}^{d})$ such that, passing to some subsequence,  
\begin{equation}
\label{eq:A.8}
\lim_{n\to \infty}u_{n}^{*}=Q 
\quad 
\mbox{weakly in $H^{1}(\mathbb{R}^{d})$ and 
strongly in $L^{p+1}(\Rd)$}. 
\end{equation}
We shall show that $Q$ becomes a minimizer for $m_{\omega}$.

	We first show $Q \not\equiv 0$. 
Suppose the contrary that $Q \equiv 0$. 
Then, it follows from (\ref{eq:A.6}) and (\ref{eq:A.8}) that, passing to some subsequence,  
\begin{equation}\label{eq:A.9}
0\ge \lim_{n\to \infty}\mathcal{N}_{\omega}(u_{n}^{*})
\ge \lim_{n\to \infty}\left\{ 
\left\| \nabla u_{n}^{*} \right\|_{L^{2}}^{2}
-\left\| u_{n}^{*} \right\|_{L^{2^{*}}}^{2^{*}}
\right\}.
\end{equation}
If $\| \nabla u_n^* \|_{L^2} \to 0$, then $\| u_n^* \|_{L^q} \to 0$ for all $2<q \leq 2^\ast$.  
By \eqref{eq:A.6} and \eqref{eq:A.7}, one has $\| u_n^* \|_{L^2} \to 0$ and 
$\widetilde{m}_{\omega} = 0$. However, this contradicts $\widetilde{m}_{\omega} > 0$ (see Lemma \ref{proposition:A.1}). Therefore, we may assume $\lim_{n\to\infty} \| \nabla u_n^* \|_{L^2} > 0$.

	Now, \eqref{eq:A.9} with the definition of $\sigma$ gives us
\[
\lim_{n\to \infty}\left\| \nabla u_{n}^{*} \right\|_{L^{2}}^{2}
\ge 
\sigma 
\lim_{n\to \infty}\left\| u_{n}^{*} \right\|_{L^{2^{*}}}^{2}
\ge 
\sigma \lim_{n\to \infty}
\left\| \nabla u_{n}^{*} \right\|_{L^{2}}^{\frac{2(d-2)}{d}}.
\]
From $\lim_{n\to\infty} \| \nabla u_n^* \|_{L^2} >0$, it follows that 
\begin{equation}\label{eq:A.10}
\sigma^{\frac{d}{2}}
\le \lim_{n\to \infty}\left\| \nabla u_{n}^{*} \right\|_{L^{2}}^2.
\end{equation}
Hence, we see from \eqref{eq:A.1}, \eqref{eq:A.7}, \eqref{eq:A.9} and \eqref{eq:A.10} that 
\[
\begin{split}
\widetilde{m}_{\omega}
&=
\lim_{n\to \infty}\mathcal{I}_{\omega}(u_{n}^{*})
\\[6pt]
&\ge 
\lim_{n\to \infty}
\left\{ 
\frac{p-1}{2(p+1)}\left\| \nabla u_{n}^{*} \right\|_{L^{2}}^{2}
+
\frac{4-(d-2)(p-1)}{2d(p+1)}\left\| u_{n}^{*} \right\|_{L^{2^{*}}}^{2^{*}}
\right\}
\\[6pt]
&\ge \frac{1}{d}\lim_{n\to \infty}\left\| \nabla u_{n}^{*} \right\|_{L^{2}}^{2}
\ge \frac{1}{d}\sigma^{\frac{d}{2}}.
\end{split}
\]
However, this contradicts \eqref{eq:A.5}. Thus, $Q \not\equiv 0$.

	Next, we shall show that $\mathcal{N}_{\omega}(Q)=0$. 
Using the Brezis-Lieb Lemma \cite{Brezis-Lieb}, we have
\begin{align}
\label{eq:A.11}
\mathcal{I}_{\omega}(u_{n}^{*})
-
\mathcal{I}_{\omega}(u_{n}^{*}-Q)
-
\mathcal{I}_{\omega}(Q)
&=o_{n}(1),
\\[6pt]
\label{eq:A.12}
\mathcal{N}_{\omega}(u_{n}^{*})-\mathcal{N}_{\omega}(u_{n}^{*}-Q)
-\mathcal{N}_{\omega}(Q)
&=o_{n}(1)
.
\end{align} 
Furthermore, \eqref{eq:A.11} together with \eqref{eq:A.7} and the positivity of $\mathcal{I}_{\omega}$ implies that
\begin{equation}\label{eq:A.13}
\mathcal{I}_{\omega}(Q)\le \widetilde{m}_{\omega}.
\end{equation}

	Let us suppose $\mathcal{N}_{\omega}(Q)<0$ and derive a contradiction. 
Note that \eqref{eq:A.13} implies that $\mathcal{I}_{\omega}(Q)=\widetilde{m}_{\omega}$. Moreover, it follows from \eqref{eq:A.2} that there exists a unique $\lambda_{0}\in (0,1)$ such that $\mathcal{N}_{\omega}(\lambda_{0}Q)=0$. Hence, we have  
\[
\widetilde{m}_{\omega}\le \mathcal{I}_{\omega}(\lambda_{0}Q)<\mathcal{I}_{\omega}(Q)=\widetilde{m}_{\omega}.
\]
This is a contradiction.

	Next, suppose that $\mathcal{N}_{\omega}(Q)>0$. 
Then, it follows from  \eqref{eq:A.6} and \eqref{eq:A.12} that $\mathcal{N}_{\omega}(u_{n}^{*}-Q)<0$ for any sufficiently large $n$. Hence, we  can take $\lambda_{n}\in (0,1)$ such that 
$\mathcal{N}_{\omega}(\lambda_{n}(u_{n}^{*}-Q))=0$. Furthermore, we see from \eqref{eq:A.7} and \eqref{eq:A.11} that  
\[
\begin{split}
\widetilde{m}_{\omega}
&\le 
\mathcal{I}_{\omega}( \lambda_{n} (u_{n}^{*}-Q))
\le \mathcal{I}_{\omega}(u_{n}^{*}-Q)
=\mathcal{I}_{\omega}(u_{n}^{*})
-
\mathcal{I}_{\omega}(Q)+o_{n}(1)
\\[6pt]
&= \widetilde{m}_{\omega} -
\mathcal{I}_{\omega}(Q)+o_{n}(1).
\end{split}
\]
Hence, we conclude that $\mathcal{I}_{\omega}(Q)=0$ and $Q \equiv 0$. 
However, this is a contradiction. Thus $\mathcal{N}_{\omega}(Q)=0$.

	Since $Q \not\equiv 0$ and $\mathcal{N}_{\omega}(Q)=0$, we have 
\begin{equation}\label{eq:A.15}
m_{\omega}
\le \mathcal{S}_{\omega}(Q)=\mathcal{I}_{\omega}(Q). 
\end{equation}
Moreover, it follows from \eqref{eq:A.11} and Proposition \ref{proposition:A.1} that 
\begin{equation}\label{eq:A.16}
\mathcal{I}_{\omega}(Q) 
\le 
\liminf_{n\to \infty}\mathcal{I}_{\omega}(u_{n}^{*})
\le \widetilde{m}_{\omega}=m_{\omega}.
\end{equation} 
Combining (\ref{eq:A.15}) and (\ref{eq:A.16}), we obtain $\mathcal{S}_{\omega}(Q)=\mathcal{I}_{\omega}(Q)=m_{\omega}$. Thus, we have proved that $Q$ is a minimizer for $m_{\omega}$. 
\end{proof}

% % % % % % % % % % % % % % % % % % % % % % % % % % % % % % % % %
% % % % % % % % % % % % % % % % % % % % % % % % % % % % % % % % %

\section{The Moser iteration}\label{section:B}

% % % % % % % % % % % % % % % % % % % % % % % % % % % % % % % % %
% % % % % % % % % % % % % % % % % % % % % % % % % % % % % % % % %

Here we state a result used in sections \ref{section:3} and \ref{section:4} 
to obtain the uniform decay estimates.

	\begin{proposition}\label{proposition:B.1} 
		Assume $d\ge 3$. 
		Let $a(x)$ and $b(x)$ be functions on $B_{4}$, 
		and let $ u \in H^1(B_4)$ be a weak solution to 
			\begin{equation*}\label{eq:B.1}
				-\Delta u + a(x) u = b (x) u \quad {\rm in} \ B_4.
			\end{equation*}
		Suppose that $a(x)$ and $u$ satisfy that 
			\begin{equation*}\label{eq:B.2}
				a(x) \geq 0 \quad {\rm for\ a.a.}\ x \in B_4, \quad 
				\int_{B_4} a(x) |u(x)v(x)| d x < \infty \quad 
				{\rm for\ each\ } v \in H^1_0(B_4). 
			\end{equation*}
		
		\begin{enumerate}
		\item[\emph{(i)}] 
			Assume that for any $\e \in (0,1)$, there exists $t_{\e}>0$ such that 
				\begin{equation*}\label{eq:B.3}
					\left\| \chi_{[|b|>t_{\e}]}b  \right\|_{L^{d/2}(B_4)} \leq \e
				\end{equation*}
		where $[|b|>t] := \left\{ x \in B_4 : |b(x)| > t \right\}$, and $\chi_A(x)$ 
		denotes the characteristic function of $A \subset \Rd$. 
		Then, for any $q \in (0,\infty)$, there exists a constant $C(d,q, t_{\e})$ such that 
			\[
				\| |u|^{q+1} \|_{H^1(B_1)} \leq C(d,q,t_{\e}) \| u \|_{L^{2^\ast}(B_4)}. 
			\]
		
		\item[\emph{(ii)}]
			Let $s>d/2$ and assume that $ b \in L^{s}(B_{4})$. 
			Then, there exists a constant $C(d,s, \|b\|_{L^{s}(B_{4})})$ such that 
			\[
				\| u \|_{L^\infty(B_1)} \le C \left(d,s, \|b\|_{L^{s}(B_{4})} \right) 
				\| u \|_{L^{2^\ast}(B_4)}.
			\]
	\end{enumerate}
	Here, the constants $C(d,q,t_{\e})$ and $C(d,s, \|b\|_{L^{s}(B_{4})})$ in \emph{(i)} and \emph{(ii)} remain bounded as long as $q$, $t_{\e}$ and $\|b\|_{L^{s}(B_{4})}$ are bounded. 
	\end{proposition}

By the assumption of Proposition \ref{proposition:B.1}, notice that 
	\begin{equation*}
		\int_{B_4} a (x) \varphi u^2 dx \geq 0 \quad \text{for all } u \in H^1(B_4)
	\end{equation*}
where $\varphi \in C^\infty_0(B_4)$ with $\varphi \geq 0$. 
Using this fact and arguing as in \cite[Proof of Proposition 2.2]{Liu-Liu-Wang} and \cite{GT} (cf. \cite{BK-79}). 
we may prove Proposition \ref{proposition:B.1}. 
Therefore, we omit the details of the proof.

\section{The Pucci-Serrin condition}\label{section:C}

In this section, we give the range of space dimension $d$ and the subcritical power $p$ for which  \cite[Theorem 1]{PS} is applicable to the case of equation \eqref{eq:1.1}. 

\begin{proposition}\label{proposition:C.1}
Let $3 \leq d \leq 6$ and assume $ \frac{4}{d-2} \leq p < \frac{d+2}{d-2}$ with $1<p$. 
Then, for any $\omega >0$, the equation \eqref{eq:1.1} admits at most one positive radial solution. 
\end{proposition}

\begin{proof}
In order to apply \cite[Theorem 1]{PS}, 
what we need to check is \cite[(2.5)]{PS}. 
In our case, this condition becomes 
	\begin{equation}\label{eq:C.1}
		\frac{d}{du} \left[ \frac{F(u)}{f(u)} \right] \geq \frac{d-2}{2d} \quad 
		{\rm for\ } u > 0, \ u \neq a
	\end{equation}
where 
	\[
		f(u) : = - \omega u + u^p + u^q, \quad 
		F(u) := - \frac{\omega}{2} u^2 + \frac{u^{p+1}}{p+1} + \frac{u^{q+1}}{q+1}, 
		\quad q := \frac{d+2}{d-2}, \quad f(a) = 0.
	\]

	We first rewrite \eqref{eq:C.1}. Since 
	\[
		\frac{d}{du}\left[ \frac{F(u)}{f(u)} \right] 
		= 1 - \frac{F(u) f'(u) }{ (f(u))^2 },
	\]
\eqref{eq:C.1} is equivalent to 
	\begin{equation}\label{eq:C.2}
		0 \leq q (f(u))^2 - (q+1) F(u) f'(u) =:g(u) \quad {\rm for\ all}\ u > 0.
	\end{equation}

	Next, we expand $g(u)$ as follows:
	\[
		g(u)
		=
		A_{2} \omega^{2} u^{2} 
		+
		A_{p+1} \omega u^{p+1}
		+
		A_{q+1} \omega u^{q+1}
		+
		A_{2p} u^{2p}
		+
		A_{p+q} u^{p+q}
	\]
where 
\begin{equation}\label{eq:C.3}
	\begin{aligned}
		A_{2}&:=\frac{q-1}{2},
		\\[6pt]
		A_{p+1}&:=\frac{(p^{2}-3p-2)q+p^{2}+p+2}{2(p+1)},
		\\[6pt]
		A_{q+1}&:=\frac{(q-1)(q-2)}{2},
		\\[6pt]
		A_{2p}&:=\frac{q-p}{p+1},
		\\[6pt]
		A_{p+q}&:=\frac{(q-p)(p+1-q)}{p+1}.
	\end{aligned}
\end{equation}
We remark that our assumption yields $ q - 1 \leq p < q$ and $2 \leq q$. 
Hence, it is easily seen that 
\begin{equation}\label{eq:C.4}
A_2 >0, \  A_{q+1} \geq 0, \  A_{2p} > 0, \ A_{p+q} \ge 0. 
\end{equation}

To show \eqref{eq:C.2}, we divide the arguments into two cases:

\smallskip

\noindent
{\bf Case 1:} $d=3$. 

\smallskip

	In this case, we have $q = 5$ and $ 4 \leq p < 5$, which implies 
	\[
		2 (p+1) A_{p+1} = 6p^2 - 14 p - 8 > 0 \quad 
		\text{for all $ 4 \leq p < 5$.}
	\]
By \eqref{eq:C.4}, \eqref{eq:C.2} holds. 

\smallskip

\noindent
{\bf Case 2:} $d=4, 5, 6$. 

\smallskip

In this case, we rewrite $g(u)$ as follows:
	\[
		g(u) = \omega^2 u^2 \left\{
		A_2 + A_{p+1} \frac{u^{p-1}}{\omega} + A_{2p} \left( \frac{u^{p-1}}{\omega} \right)^2 \right\}
		+ A_{q+1} \omega u^{q+1} + A_{p+q} u^{p+q}.
	\]
By \eqref{eq:C.4}, it suffices to show 
	\[
		Q(r) := A_2 + A_{p+1} r + A_{2p} r^2 \geq 0 \quad 
		\text{for each $r \geq 0$ and $q-1 \leq p < q$}. 
	\]

When $d=4$, one has $q = 3$ and 
	\[
		Q(r) = 1 + \frac{2(p^2-2p-1)}{p+1} r + \frac{3-p}{p+1} r^2.
	\]
If $p^{2}-2p-1 \ge 0$, then $Q(r)> 0$ for all $r\ge 0$. 
On the other hand, if $p^{2}-2p-1 < 0$, then 
we obtain $2 \leq p < 1 + \sqrt{2} = :p_0$ and simple computations give 
\begin{equation}\label{eq:C.5}
\min_{r\ge 0}{Q(r)}= 
1 - \frac{(p^2-2p-1)^2}{(p+1)(3-p)}.
\end{equation}
Set 
	\[
		h(p) := (p+1)(3-p) - (p^2 - 2p - 1)^2 
		= -p^4 + 4p^3 - 3p^2 -2 p +2.
	\]
Note that 
	\[
		h'(p) = -4p^3 + 12p^2 - 6p -2, \quad 
		h''(p) = -12p^2 + 24 p - 6 < 0 \quad {\rm in} \ [2,p_0].
	\]
We also observe that 
	\[
		h'(2) = 2 >0, \quad h'(p_0) = -2p_0 + 2 < 0, \quad 
		h(2) = 2 = h(p_0) > 0.
	\]
Hence, $h(p) > 0$ in $[2,p_0]$ and by \eqref{eq:C.5}, we have 
$\min_{r \geq 0} Q(r) \geq 0$ and \eqref{eq:C.2}.

	When $d=5$, we see
	\[
		Q(r) = \frac{1}{3} \left\{ 
		2 + \frac{5p^2 - 9p - 4}{p+1} r 
		+ \frac{7-3p}{p+1} r^2
		 \right\} .
	\]
Remark that $5p^2 -9p - 4 < 0$ is equivalent to 
$9 - \sqrt{161} < 10 p < 9 + \sqrt{161}$ and 
also that $ \frac{4}{3}< \frac{9+\sqrt{161}}{10} < \frac{7}{3}$. 
Hence, if $ \frac{9+\sqrt{161}}{10} \leq p < \frac{7}{3}$, 
then $Q(r) \geq 0$ for all $ r \geq 0$.

	When $\frac{4}{3} \leq p < \frac{9+\sqrt{161}}{10}$, observe that  
	\[
	\min_{r \geq 0} Q(r) 
	= \frac{1}{3} \left\{ 2 - \frac{(5p^2 - 9p - 4)^2}{4(7-3p)(p+1)} \right\}.
	\]
Since 
	\[
		\begin{aligned}
			8(7-3p)(p+1) - (5p^2-9p-4)^2 
			&= - 25p^4 + 90 p^3 - 65p^2 - 40p + 40
			\\
			&=(p-1)^2 ( -25p^2 +40p + 40 ) =: (p-1)^2 h(p),
		\end{aligned}
	\]
by $\frac{9+\sqrt{161}}{10} < \frac{9+13}{10} = \frac{11}{5}$ and 
$h(\frac{11}{5}) = 7 > 0$, we obtain 
$h(p) \geq 0$ for every $p \in [ \frac{4}{3} , \frac{9+\sqrt{161}}{10} ]$. 
Thus, $\min_{r\geq 0} Q(r) \geq 0$ and \eqref{eq:C.2} holds.

	When $d=6$, we observe 
	\[
		Q(r) = \frac{1}{2} \left\{ 1 - \frac{(3p+1)(2-p)}{p+1} r + \frac{2(2-p)}{p+1} r^2 \right\}, 
		\ 
		\min_{r \geq 0} Q(r) = 
		\frac{1}{2} \left\{ 	1 -  \frac{(2-p)(3p+1)^2}{8(p+1)} \right\}
			 .
	\]
Setting $h(p) := 8(p+1) - (2-p) (3p+1)^2$ for $1 \leq p < 2$, we obtain 
	\[
		h(1) = 0, \quad 
		h'(p) = 27p^2 - 24 p -3 \geq 0 \quad \text{for each $1 \leq p <2$}.
	\]
Hence, $h(p) \geq 0$ for all $1 \leq p < 2$ and \eqref{eq:C.2} holds. 
\end{proof}

\begin{remark}\label{remark:C.1}
When $d \geq 7$ and $ 1 < p < \frac{d+2}{d-2} = q$, 
condition \eqref{eq:C.1} is not satisfied for $\omega \gg 1$. 
In fact, we have $1 < q < 2$ and $A_{q+1} < 0$ in \eqref{eq:C.3}. 
Fix an $\alpha \in ( \frac{1}{q-1}, \frac{1}{p-1} )$ and observe that 
	\[
		 \alpha ( p + q) < 1 + \alpha ( q + 1), \quad 
		2 + 2\alpha < 1 + \alpha ( q + 1).
	\]
Noting also $1+(p+1) \alpha < 1+(q+1)\alpha$ and $2 p \alpha < (p+q) \alpha$, 
 we see that 
	\[
		\begin{aligned}
			g( \omega^{\alpha} ) 
			&= A_2 \omega^{2+2\alpha} + A_{p+1} \omega^{1+ (p+1)\alpha} 
			+ A_{q+1} \omega^{ 1 + (q+1) \alpha } + A_{2p} \omega^{2\alpha p} 
			+ A_{p+q} \omega^{(p+q)\alpha}
			\\
			&= \omega^{1+(q+1)\alpha} \left( A_{q+1} + o(1) \right).
		\end{aligned}
	\]
Since $A_{q+1} < 0$, we obtain $g(\omega^\alpha) < 0$ for $\omega \gg 1$ and 
\eqref{eq:C.1} is not satisfied. 

It is worth noting that for any $\Phi_\omega \in \Go$ we have 
$\Phi_\omega (0) = \| \Phi_\omega \|_{L^\infty} \sim \omega^{\frac{1}{p-1}}$ 
by Lemma \ref{proposition:4.1}. Hence, \eqref{eq:C.1} breaks down 
even in the interval $[0,\Phi_\omega(0)]$. 
\end{remark}

%%%%%%%%%%%%%%%%%%%%%%%%%%%%%%%%%%%%%%%%%%%%
%%%%%%%%%%%%%%%%%%%%%%%%%%%%%%%%%%%%%%%%%%%%
%%%%%%%%%%%%%%%%%%%%%%%%%%%%%%%%%%%%%%%%%%%%

\subsection*{Acknowledgement}

S.I. is partially supported by NSERC Discovery grant \# 371637-2014, and also acknowledges the kind hospitality of Tsuda University, Japan. The work of N.I. was supported by JSPS KAKENHI Grant Number JP16K17623 and 
JP17H02851. 
The work of H.K. was supported by JSPS KAKENHI Grant Number 
JP17K14223. The work of H.N. was supported by JSPS KAKENHI Grant Number 17H02859 and 15K13450.

%%%%%%%%%%%%%%%%%%%%%%%%%%%%%%%%%%%%%%%%%%%%%%%%%%%%%%%%%%%%%%%%%%%%%%%%%%%%%%%

\bibliographystyle{plain}

%%%%%%%%%%%%%%%%%

\end{document}